\documentclass{amsart}

\usepackage{graphicx}
\usepackage{bbm}
\usepackage[caption=false]{subfig}
\usepackage{epstopdf}

\newtheorem{theorem}{Theorem}[section]
\newtheorem{lemma}[theorem]{Lemma}

\theoremstyle{definition}
\newtheorem{definition}[theorem]{Definition}
\newtheorem{example}[theorem]{Example}

\theoremstyle{remark}
\newtheorem{remark}[theorem]{Remark}

\numberwithin{equation}{section}

\begin{document}

\title{Risk-sensitive linear-quadratic-Gaussian graphon mean-field games}

\author{Tian Chen}
\address{School of Mathematics, Shandong University, Jinan, 250100, China; and School of Mathematics and Statistics, Carleton University, Ottawa, ON  K1S 5B6, Canada.}
\email{chentian43@sdu.edu.cn}
\thanks{The first author was supported in part by the National Science Foundation of China under Grant 12401583.}

\author{Minyi Huang}
\address{School of Mathematics and Statistics, Carleton University, Ottawa, ON K1S 5B6, Canada. Corresponding author (M.H.). }
\email{mhuang@math.carleton.ca}
\thanks{The second author was supported in part by Natural Sciences and Engineering Research Council of Canada.}





\begin{abstract}
This paper investigates a class of linear-quadratic-Gaussian risk-sensitive graphon mean-field games, involving an asymptotically infinite population of heterogeneous agents distributed across an asymptotically infinite network, where each agent aims to minimize an exponential cost functional reflecting its risk sensitivity. Following the Nash certainty equivalence methodology,   an auxiliary risk-sensitive optimal control problem is constructed and further combined with a consistency condition  to determine decentralized strategies of the agents. The well-posedness of the resulting graphon mean-field game equation system, consisting of a family of fully coupled forward-backward differential equations, is established by a fixed point approach under a contraction condition, and by the method of continuity under an operator monotonicity condition, respectively.  To prove the $\varepsilon$-Nash equilibrium property of the obtained  decentralized strategies,  one faces significant challenge since the usual $L^2$ error estimates on mean-field approximations are no longer adequate due to unboundedness of the integrand in the exponentiated cost. The proof will be accomplished by establishing certain exponentiated error estimates instead of $L^2$ error estimates.  Finally, a numerical example is provided to illustrate our results.
\end{abstract}

\maketitle

\section{Introduction}
Mean-field game (MFG) theory has rapidly developed into a robust framework for addressing non-cooperative dynamic decision-making in large populations of comparably small agents \cite{HMC2006,HCM2007,LL2007}. Most existing analyses are based on two approaches: the direct approach and the fixed point approach. The direct approach solves an $N$-player game, resulting in a large-scale coupled equation system. As the number of agents $N\rightarrow \infty$, this system reduces to a Hamilton--Jacobi--Bellman (HJB) equation coupled with a Fokker--Planck--Kolmogorov (FPK) equation \cite{LL2007,CDLL2019,L2016}. Alternatively, the fixed point approach starts by assuming a known mean-field effect and derives the best response control law of a representative agent. Subsequently, the above mean-field is required to be regenerated by the closed-loop dynamics of an infinite population.  This solution procedure is formalized as a fixed point problem \cite{HCM2007,BFY2013,CD2018}, and is designated as the Nash certainty equivalence methodology in \cite{HMC2006}. Recently, \cite{HZ2019} analyzed the exact relationship between these two approaches in the setup of linear-quadratic-Gaussian (LQG) MFGs.  MFG theory has found wide applications in areas such as economics and finance \cite{ET2015,FZ2023}, production output adjustment \cite{WH2019}, and traffic management \cite{BZ2016}, among others.

Graphon mean-field games (GMFGs) as introduced in \cite{CH2021} significantly extend the classical MFG framework by incorporating agents distributed across large networks.  As networks increase in size and complexity, graphon theory becomes a powerful tool to study their limits, enabling analysis of dense networks \cite{LS2006,BCLSV2008,BCLSV2012,L2012}. Graphons, as measurable functions representing the limit of increasingly large graphs, provide a rigorous mathematical framework for studying the behavior of large-scale networks. This framework is particularly studied for fields such as economics, control theory, social networks, and large-scale interacting particle systems, where agents demonstrate heterogeneous and complex interactions \cite{D2017,CCGL2022,AC2022,GCH2023,BCW2023,FGCH2024}. GMFGs  enable researchers to study the limiting behavior of systems with increasing populations of agents, capturing both local and global network interaction effects.

The aforementioned literature on MFGs or GMFGs has primarily focused on risk-neutral cost functionals, often neglecting agents' attitudes towards risk. However, in many practical scenarios, risk sensitivity plays a crucial role in decision-making, as recognized very early in \cite{HM1972}. Jacobson \cite{J1973} introduced a finite horizon risk-sensitive optimal control problem in an LQG setting, where the agent's cost is specified as an exponential function of the accumulated cost over time. Since the early work in \cite{HM1972,J1973}, this class of control problems has been extensively studied in the literature; see \cite{W1990,FM1995,FD1997,PB1996,BFN1998,LZ2005,D2013}. Within the framework of MFGs, risk-sensitive costs were adopted in \cite{TZB2013}, which  specified the mean-field equilibrium strategies using a corresponding pair of  HJB and FPK equations. Subsequently, Moon and Ba\c{s}ar \cite{MB2019} extended this line of research by employing the maximum principle to analyze nonlinear risk-sensitive MFG models.  Risk-sensitive  MFGs with major-minor agents were studied in \cite{CLX2023,LFB2023}, addressing the interaction structure between the major and minor agents. In the setting of discrete-time Markov decision processes, related results can be found in \cite{SBR2020}.  Recently, Wang and Huang \cite{WH2024} considered LQG risk-sensitive MFGs and proved an $O(1/N)$-Nash equilibrium theorem by applying a re-scaling technique to a  family of coupled-Riccati equation systems with increasing dimensions.

In this paper, we formulate a class of linear-quadratic-Gaussian GMFGs with risk-sensitive cost functionals. Using the Nash certainty equivalence methodology \cite{HMC2006,HCM2007}, we solve the risk-sensitive best response control problem of a representative agent after approximating its graphon coupling term by a deterministic function,  and subsequently impose a consistency condition on the proposed limiting function. This solution procedure leads to  the GMFG equation system, as a family of fully coupled forward-backward equations, whose well-posedness is established via a fixed point approach under a contraction condition, and by the method of continuity under an operator monotonicity condition, respectively. In contrast to similar conditions in the literature (see \cite{HP1995,PW1999}),  our monotonicity condition involves operators acting on an infinite dimensional space. Furthermore, by a novel method of exponentiated error estimate, we prove that the set of decentralized strategies computed via the GMFG equation system constitutes an $\varepsilon$-Nash equilibrium for the finite population of agents. When the running cost is unbounded as in the LQG case, one encounters significant difficulty proving that the set of  decentralized strategies obtained from the risk-sensitive MFG is an $\varepsilon$-Nash equilibrium.  In this case, the usual $L^2$ error estimates for mean-field approximations is inadequate for handling  exponentiated cost functionals; see Remark \ref{remark47} later on. Moreover,  the method of performance estimates in \cite{WH2024} is not applicable to the GMFGs due to  lack of symmetry of the agents.   To establish the $\varepsilon$-Nash equilibrium theorem, we take a different route by deriving new approximation error estimates of exponential type  without imposing any additional restrictions as used in the literature (see discussions in \cite[Section 4.4]{LFB2023}).

The main contributions of this paper are summarized as follows:

\begin{itemize}
  \item  We introduce a novel risk-sensitive GMFG model, where the large population of agents is distributed over dense networks.

  \item The solution of GMFG is characterized by a fully coupled forward-backward equation system parameterized by the nodal index. The unique solvability of the GMFG  equation system is analyzed by two methods: (i) the fixed point method under a contraction condition, (ii) the method of continuity under a new operator monotonicity condition.

  \item Based on the GMFG equation system, we construct decentralized feedback strategies in the $N$-player model. A novel method via exponentiated error estimates is applied to prove that the set of decentralized strategies constitutes an $\varepsilon$-Nash equilibrium.

\end{itemize}

The paper is  organized as follows. Section \ref{sec:problem}  formulates the risk-sensitive GMFG. Section \ref{sec:design} solves the best response  control problem in a graphon mean-field limit model. By imposing a consistency condition,  Section \ref{sec:consistency} derives the GMFG equation system and investigates its solvability  by two approaches: the fixed point method and the method of continuity. In Section \ref{sec:Nash}, we  prove an $\varepsilon$-Nash equilibrium theorem for the obtained decentralized strategies. Section \ref{sec:numerical} presents a numerical example. Section \ref{sec:conclusion} concludes the paper.

\section{Problem formulation and preliminary}
\label{sec:problem}

We use $\mathbb{R}^m$ to denote the $m$-dimensional Euclidean space, consisting of column vectors, with norm  $|\cdot|$ and inner product  $\langle \cdot,\cdot\rangle$. Let $D^{ \top}$ (resp., $D^{-1}$) stand for the transpose (resp., inverse) of a matrix  $D$, and $\mathbb S^m$ for the set of symmetric real $m\times m$ matrices. If  $D\in \mathbb{S}^n$ is positive definite (resp., positive semi-definite), we write $D> 0$ (resp., $ D\ge 0$). Moreover, if an $\mathbb S^n$-valued deterministic function $D$ defined on $[0,T]$ is uniformly positive definite, we write $D\gg 0$. For two sequences of number $\{a_k\}_{k\geq 1}$ and $\{b_k\}_{k\geq 1}$ with $b_k>0$, if $\lim_{k\rightarrow \infty}\frac{a_k}{b_k}=0$, we write $a_k=o(b_k)$.

Let $(\Omega,\mathcal{F},\mathbb F,\mathbb{P})$ be a complete filtered probability space with  filtration $\mathbb F:=\{\mathcal F_t\}_{0\leq t\leq T}$. Given Hilbert space $\mathbb H$ and interval $[0,T]$, let  $L^{\infty}([0,T]; \mathbb{H})$ stand for the space of all $\mathbb H$-valued essentially bounded functions defined on $[0,T]$, $C([0,T];\mathbb H)$ for the space of all $\mathbb H$-valued continuous functions, $L^2_{\mathbb{F}}([0,T];\mathbb{H})$ for the space of all $\mathbb H$-valued, $\mathcal{F}_t$-adapted, square-integrable processes, and $S^2_{\mathbb{F}}([0,T];\mathbb{H})$ for the space of all $\mathbb H$-valued, $\mathcal{F}_t$-adapted continuous processes $\phi$ such that $ \mathbb{E}[\sup_{0\leq t\leq T}|\phi(t)|^2] < \infty$.

\subsection{Preliminary}\label{sec:sub:preliminary}

The basic idea of the theory of graphons is that the edge structure of each finite cardinality network is represented by a step function density on the unit square in $\mathbb R^2$ on which the cut norm and cut metric (see \cite{L2012}) are introduced.  The set of finite graphs endowed with the cut metric then gives rise to a metric space, and the completion is the space of graphons. Let $\bf{G}^{sp}_0$ be the linear space of bounded  symmetric Lebesgue measurable functions $W:[0,1]^2 \rightarrow \mathbb R$, which are called kernels. The space $\bf{G}^{sp}$ of graphons is a subset of  $\bf{G}^{sp}_0$ and consists of kernels $W:[0,1]^2 \rightarrow [0,1]$, which can be interpreted as weighted graphs on the vertex set $[0,1]$.

In this paper, we start the modeling of the game of a finite population based on a finite graph. Specifically, the population resides on a weighted finite graph $G_N$ with a set of nodes $\mathcal V_N=\{1,2,\cdots, N\}$ and weights $g_{ij}^N\in [0,1]$ for $(i,j)\in \mathcal V_N\times \mathcal V_N$, where a value $g_{ii}^N$ is still assigned in the case $i=j$. We call $g_i^N=(g_{i1}^N, g_{i2}^N,\cdots, g_{iN}^N)$ a section of $g^N$ at $i$. Moreover, each node $l$ is occupied by exactly one agent, denoted by ${\mathcal A}_l$. In fact, it brings about no essential difficulty to consider the case where each node $l$ is occupied by a subpopulation.  Our further analysis is based on the convergence of $g^N$ to a graphon limit $g$. We may naturally identify $(g_{ij}^N)_{1\leq i,j\leq N}$ with a graphon $g^N(\alpha,\beta)$ as a step function defined on $[0,1]\times [0,1]$ (see \cite{L2012}).  We define the section of $g$ at $\alpha$ by $g_{\alpha}:\beta\mapsto g({\alpha,\beta})$, $\beta\in [0,1]$. We partition $[0,1]$ into $N$ subintervals of equal length. Here $I_1^N=[0,1/N]$, $I_l^N=((l-1)/N,l/N]$ for $2\leq l\leq N$. When it is clear from the context, we omit the superscript $N$ and write $I_l$. To relate the agents to the graphon vertex set $[0,1]$, we let the $i$-th agent correspond to $I_i$.

\subsection{Graphon MFG}
For the risk-sensitive GMFG, consider a population of $N$ agents $\{\mathcal A_i\}_{1\leq i\leq N}$. The dynamics of agent $\mathcal A_i$ residing  at node $i$ are governed by the following SDE:
\begin{equation}\label{state}
  \begin{aligned}
    &\mathrm{d}x_i(t)=[A(t)x_i(t)+B(t)u_i(t)+D(t)x^{(N)}_i(t)]\mathrm{d}t+\sigma(t)\mathrm{d}w_i(t),\quad
    x_i(0)=\xi_i,
  \end{aligned}
\end{equation}
where $x_i(t)\in \mathbb R^n$ is the state, $u_i(t)\in \mathbb R^m$ is the control, and $w_i, 1\leq i\leq N$, are independent $d$-dimensional standard Brownian motions. Here, $x^{(N)}_i=\frac{1}{N}\sum_{j=1}^N g_{ij}^Nx_j$  represents the weighted state-average of all agents with respect to node $i$.

Now we introduce two types of state feedback strategies. Let the functions $\varphi(t,x_1,\cdots, x_N):[0,T]\times \mathbb R^n \times \cdots\times \mathbb R^n\rightarrow \mathbb R^m$ and $\psi(t,x):[0,T]\times \mathbb R^n\rightarrow \mathbb R^m$ be continuous in $t$, and Lipschitz continuous in $(x_1,\cdots,x_N)$ and $x$, respectively. For agent $\mathcal A_i$, we introduce
\begin{equation}\label{admissibleset}
    \mathcal U_c^i=\{u_i\ |\ u_i=\varphi(t,x_1,\cdots,x_N)\},\quad\quad \mathcal U_d^i=\{u_i\ |\ u_i=\psi(t,x_i)\}
\end{equation}
as the set of centralized (state) feedback strategies and the set of decentralized (state) feedback strategies, respectively. For simplicity, let $u=(u_1,\cdots, u_N)$ be the set of strategies of $N$ agents and $u_{-i}=(u_1,\cdots, u_{i-1}, u_{i+1},\cdots, u_N)$ be the set of strategies of all agents other than $\mathcal A_i$. Note that while the coefficients are dependent on the time variable $t$, in further analysis, the variable $t$ will usually be suppressed if no confusion occurs.

The risk-sensitive cost functional of agent $\mathcal A_i$ takes the form
\begin{equation}\label{cost}
  \begin{aligned}
    \!\mathcal J_i(u_i,u_{-i}) \!=\! \mathbb E\Big[\exp\Big( {\gamma}\Big\{\!\int_0^T\!\!\big[\|x_i\!-\!\Gamma x^{(N)}_i\|_{Q}^2\! +\!\|u_i\|^2_{R}\big]\mathrm{d}t \!+\!\|x_i(T)\!-\!\Gamma_f x^{(N)}_i(T)\|_{Q_f}^2\Big\}\Big)\Big],
  \end{aligned}
\end{equation}
where we denote $\|x\|_{Q}^2=\langle Qx, x\rangle $, and both $Q$ and $R$ are matrix functions of $t$. Here $\gamma>0$ is a given constant representing the risk sensitivity parameter.

For convenience of describing the asymptotic behavior of the agents and the networks, we need to relate the node of each agent to the continuum nodal set $[0,1]$. Without loss of generality, we associate agent ${\mathcal A}_i$, $1\le i\le N$,
with the subinterval $I_i$ in the partition of $ [0,1]$, which in turn determines the corresponding  step function $g^N(\cdot, \cdot)$.
We introduce the following assumptions.

\textbf{(H1)} The initial states $\xi_i, 1\leq i\leq N$, are independent random variables taking values in a compact set $S_0^x\subset \mathbb R^n$, and are independent of $\{w_i, 1\leq i\leq N\}$. There exists a continuous function $m^x_\alpha $ defined on $ [0,1]\ni \alpha $, such that $\lim_{N\rightarrow \infty}\sup_{0\leq \alpha\leq 1}|\mathbb E\xi_{\alpha}^N $ $-m^x_\alpha|=0$,  where $\xi_{\alpha}^N=\sum_{i=1}^N\mathbbm 1_{I_i}(\alpha)\xi_i$.

\textbf{(H2)} $A, D \in L^{\infty}([0,T];\mathbb R^{n\times n})$, $B\in L^{\infty}([0,T];\mathbb R^{n\times m})$, $\sigma\in L^{\infty}([0,T];\mathbb R^{n\times d})$.

\textbf{(H3)} $Q\in L^{\infty}([0,T];\mathbb S^n)$, $R\in L^{\infty}([0,T];\mathbb S^m)$, $Q_f\in \mathbb S^n$, $\Gamma, \Gamma_f\in \mathbb R^{n\times n}$, and $Q\geq 0$, $R\gg 0$, $Q_f\geq 0$.

\textbf{(H4)} $BR^{-1}B^{\top}-{2}{\gamma}\sigma\sigma^{\top}\geq   0$ for all $t\in [0,T]$.

\textbf{(H5)} For any bounded, measurable function $h(\beta)$, the function $\int_0^1g(\alpha,\beta)h(\beta)\mathrm{d}\beta$ is continuous in $\alpha\in [0,1]$.

\begin{remark}
   Assumption (H5) is fulfilled under mild conditions on the graphon $g$. We identify $g(\alpha,\cdot)$ as an element in the dual space of $L^{\infty}[0,1]$, i.e., $g(\alpha,\cdot)\in (L^{\infty}[0,1])'$; see \cite[p. 118, Example 5]{Y1971}. Denote $\varphi(\alpha)=g(\alpha,\cdot)$ as a $(L^{\infty}[0,1])'$-valued mapping. Then (H5) is eqivalent to weak convergence of $\varphi(\alpha_k)$ to $\varphi(\alpha)$ whenever $\alpha_k\rightarrow \alpha$. A sufficient condition to ensure (H5) is that $g(\alpha_k,\cdot)\rightarrow g(\alpha,\cdot)$ in $L^1[0,1]$ as $\alpha_k\rightarrow \alpha$. Typical examples of graphons satisfying this condition include the uniform attachment graphon  $g_1(\alpha,\beta)=1-\max(\alpha,\beta)$ (see \cite[p. 188]{L2012}) and the half graphon $g_2(\alpha,\beta)=\mathbbm 1_{\{\beta\geq \alpha+0.5\text{ or } \alpha\geq \beta+0.5\}}$ (see \cite[p. 17]{L2012}). The half graphon is not continuous on $[0,1]^2$.
\end{remark}

Under (H1)-(H3), one can easily check that SDE \eqref{state} admits a unique solution $x_i\in S^2_{\mathbb F}([0,T];\mathbb R^n)$ for any control $u_i\in L^2_{\mathbb F}([0,T];\mathbb R^m)$ and the cost functional \eqref{cost} is well-defined.

A basic solution notion for \eqref{state}-\eqref{cost} is a Nash equilibrium $(u^*_1,\cdots, u_N^*)$, where each $u_i^*$ belongs to $\mathcal U_c^i$, corresponding to closed-loop perfect state (CLPS) information \cite{BO1998}.  However, such a solution with its associated information pattern is impractical when the system consists of a large number of heterogeneous agents.

\section{The design of decentralized strategies}\label{sec:design}
To design decentralized strategies, we need to study an auxiliary best response control problem with an infinite population and a graphon limit $g$ as an approximation of $g^N$ (the nature of the approximation will be made exact by (H7) in Section \ref{sec:Nash}). For an $\alpha$-agent ${\mathcal A}_\alpha$ situated at vertex $\alpha$, denote the graphon weighted mean state by $ z_{\alpha}\in C([0,T];\mathbb R^n)$, which is intended to approximate $x_i^{(N)}$ when the nodal location of $\mathcal A_i$ is approximated by $\alpha$. Let $w_{\alpha}$ be the standard Brownian motion associated with the $\alpha$-agent. Then the $\alpha$-agent has dynamics
\begin{equation}\label{lstate}
  \begin{aligned}
    &\mathrm{d}x_{\alpha}=(Ax_{\alpha}+Bu_{\alpha}+D z_{\alpha})\mathrm{d}t+\sigma\mathrm{d}w_{\alpha}, \quad x_{\alpha}(0)=\xi_{\alpha},\quad  \alpha\in [0,1],
  \end{aligned}
\end{equation}
and cost functional
\begin{equation}\label{lcost}
  \begin{aligned}
    J_{\alpha}(u_{\alpha})\!=\! \mathbb E\Big[\!\exp\Big( {\gamma}\Big\{\int_0^T\!\big[\|x_{\alpha}\!-\!\Gamma z_{\alpha}\|_{Q}^2\!+\!\|u_{\alpha}\|^2_{R}\big]\mathrm{d}t \!+\!\|x_{\alpha}(T)\!-\!\Gamma_f  z_{\alpha}(T)\|_{Q_f}^2\Big\}\Big)\Big].
  \end{aligned}
\end{equation}
The Brownian motion $w_{\alpha}$ and the initial state $\xi_{\alpha}$ are  independent. The auxiliary best response (BR) control problem is formulated as follows.

\textbf{Problem (BR).} For the $\alpha$-agent $\mathcal A_{\alpha}$ in \eqref{lstate}-\eqref{lcost},  find a feedback control law $\bar u_{\alpha}$ such that \vspace{-1em}
\begin{equation*}
  J_{\alpha}(\bar u_{\alpha})=\inf_{u_{\alpha}} J_{\alpha}(u_{\alpha}).
\end{equation*}

For any given $t\! \in\!  [0,T]$ and a deterministic initial condition $x_{\alpha}(t)\! =\! {\bf x}\! \in\!  \mathbb R^n$, define
\begin{equation*}
    J_{\alpha}(t,{\bf x};u_{\alpha})\!=\! \mathbb E\Big[\!\exp\Big( {\gamma}\Big\{\int_t^T\!\big[\|x_{\alpha}\!-\!\Gamma z_{\alpha}\|_{Q}^2\!+\!\|u_{\alpha}\|^2_{R}\big]\mathrm{d}s \!+\!\|x_{\alpha}(T)\!-\!\Gamma_f  z_{\alpha}(T)\|_{Q_f}^2\Big\}\Big)\Big].
\end{equation*}
and write the value function in the form
\begin{equation*}
    \inf_{u_{\alpha}} J_{\alpha}(t,{\bf x}; u_{\alpha})=e^{\gamma V^{\alpha}(t,{\bf x})},
\end{equation*}
where the function $V^{\alpha}$ is to be determined.
We write the gradient of a function as a row vector.
Applying dynamic programming (see \cite[Eqn. (1.4)]{BFN1998}), we have
\begin{equation}\label{HJB}
\left\{
  \begin{aligned}
    &\!-\!\frac{\partial V^{\alpha}(t,{\bf x})}{\partial t}\!=\!\inf_{u\in \mathbb R^m}\Big\{\frac{\partial V^{\alpha}(t,{\bf x})}{\partial {\bf x}}[A{\bf x}\!+\!Bu\!+\!Dz_{\alpha}(t)]\!+\|{\bf x}-\Gamma z_{\alpha}(t)\|_{Q}^2\!+\!\|u\|^2_{R}\Big\}\\
    &\!\qquad\qquad\qquad\qquad\!+\!\frac{1}{2}\text{Tr}\Big(\sigma\sigma^{\top} \frac{\partial^2V^{\alpha}(t,{\bf x})}{\partial {\bf x}^2}\Big)\!+\!\frac{\gamma}{2}\Big|\sigma^{\top}\Big(\frac{\partial V^{\alpha}(t,{\bf x})}{\partial {\bf x}}\Big)^{\top}\Big|^2,\\
    &V^{\alpha}(T,{\bf x})\!=\!\|{\bf x}\!-\!\Gamma_f  z_{\alpha}(T)\|_{Q_f}^2,\quad (t,{\bf x})\in [0,T]\times \mathbb R^n,\quad \alpha\in [0,1],
  \end{aligned}
  \right.
\end{equation}
  and the optimal control law is given by
\begin{equation*}
    \bar u_{\alpha}(t)=-\frac{1}{2}R^{-1}(t)B^{\top}(t)\Big(\frac{\partial V^{\alpha}(t,{\bf x})}{\partial{\bf x}}\Big)^{\top}.
\end{equation*}

We write $V^{\alpha}(t,{\bf x})={\bf x}^{\top}\Pi(t){\bf x}+2{\bf x}^{\top}S_{\alpha}(t)+r_{\alpha}(t)$. Then by \eqref{HJB}, we derive the following Riccati equation
\begin{equation}
    \begin{aligned}\label{Phi}
      &\dot{\Pi}+\Pi A+A^{\top}\Pi-\Pi(BR^{-1}B^{\top}-{2}{\gamma}\sigma\sigma^{\top})\Pi+Q=0,\quad
      \Pi(T)=Q_f.
    \end{aligned}
  \end{equation}
  Under (H2)-(H4), Riccati equation \eqref{Phi} admits a unique solution $\Pi\in C([0,T];\mathbb S^n)$; see \cite[pp. 186-190]{AFIJ2003}. By \eqref{HJB}, we further derive the following ordinary differential equations (ODEs):
\begin{equation}\label{Sigma1}
    \begin{aligned}
      &\dot{S}_{\alpha}\!+\!\big(A^{\top}\!-\!\Pi BR^{-1}B^{\top}\!+\!{2}{\gamma}\Pi\sigma\sigma^{\top}\big)S_{\alpha}\!-\!(Q\Gamma\!-\!\Pi D) z_{\alpha}\!=\!0,
    \end{aligned}
  \end{equation}
  where $S_{\alpha}(T)=-Q_f{\Gamma}_fz_{\alpha}(T)$, and
  \begin{equation}\label{rho}
    \begin{aligned}
      &\dot{r}_{\alpha}\!-\!S^{\top}_{\alpha}\big(BR^{-1}B^{\top}\!-\!{2}{\gamma}\sigma\sigma^{\top}\big)S_{\alpha}\!+\!2z^{\top}_{\alpha} D^{\top}S_{\alpha}\!+\!z^{\top}_{\alpha}\Gamma^{\top} Q\Gamma z_{\alpha}\!+\!\text{Tr}(\sigma\sigma^{\top}\Pi)\!=\!0,
    \end{aligned}
  \end{equation}
  where $r_{\alpha}(T)\!=\!z^{\top}_{\alpha}(T){\Gamma}_f^{\top}Q_f{\Gamma}_fz_{\alpha}(T)$. Under (H2)-(H4), the ODE system \eqref{Sigma1}-\eqref{rho} admits a unique solution $(S_{\alpha},r_{\alpha})\in C([0,T];\mathbb R^n)\times C([0,T];\mathbb R)$. Then we obtain the candidate optimal control 
\begin{equation*}
   \bar u_{\alpha}=-R^{-1}B^{\top}\Pi x_{\alpha}-R^{-1}B^{\top}S_{\alpha}.
\end{equation*}

Now, we show that $\bar u_{\alpha}$ is indeed the optimal control of Problem (BR). Applying It\^o's formula to $\langle \Pi x_{\alpha}, x_{\alpha}\rangle +2\langle S_{\alpha}, x_{\alpha}\rangle$ and using the completion of squares technique, we have
\begin{equation*}\begin{aligned}
    J_{\alpha}( u_{\alpha})&=\mathbb E\big[e^{\gamma \{\xi_{\alpha}^{\top}\Pi(0)\xi_{\alpha}+2\xi_{\alpha}^{\top}S_{\alpha}(0)+r_{\alpha}(0) +\int_0^T\|u_{\alpha}+R^{-1}B^{\top}(\Pi x_{\alpha}+S_{\alpha})\|_R^2\mathrm{d}t \}}\big].
\end{aligned}\end{equation*}
Since $\gamma>0$,  $J_{\alpha}$ attains its minimum when $u_{\alpha}=\bar u_{\alpha}$, which implies that $\bar u_{\alpha}$ is indeed the optimal control. We summarize the above results into the following theorem.
\begin{theorem}\label{baru}
  Under (H1)-(H4),  the optimal control law of Problem (BR) is given by
  \begin{equation}\label{law}
    \bar u_{\alpha}(t)=-R^{-1}(t)B^{\top}(t)\Pi(t)  x_{\alpha}(t)-R^{-1}(t)B^{\top}(t)S_{\alpha}(t),\qquad \alpha\in [0,1],
  \end{equation}
  for which the closed-loop state process $x_{\alpha}$ satisfies
  \begin{equation}\label{optimalstate}
    \begin{aligned}
      &\mathrm{d} x_{\alpha}\!=\![(A\!-\!BR^{-1}B^{\top}\Pi) x_{\alpha}\!-\!BR^{-1}B^{\top}S_{\alpha}\!+\!Dz_{\alpha}]\mathrm{d}t \!+\!\sigma\mathrm{d}w_{\alpha}, \quad
      x_{\alpha}(0)\!=\!\xi_{\alpha}.
    \end{aligned}
  \end{equation}
Moreover, the optimal cost is given by
  \begin{equation}\label{optimalcost1}
    \begin{aligned}
      J_{\alpha}(\bar u_{\alpha})= \mathbb E\big[\exp\big(\gamma\big\{\xi_{\alpha}^{\top}\Pi(0)\xi_{\alpha}+2\xi_{\alpha}^{\top}S_{\alpha}(0)+r_{\alpha}(0)\big\}\big)\big].
    \end{aligned}
  \end{equation}
\end{theorem}

\section{Consistency condition}\label{sec:consistency}
The best response control problem of Section \ref{sec:design} has been solved by assuming each $z_{\alpha}(\cdot) $ to be a known function defined on $[0,T]$. Following \cite{CH2021,FGCH2024}, we specify $z_{\alpha}$ by imposing a consistency condition such that $z_{\alpha}(t)$ is regenerated by the graphon weighted average of the individual means:
\begin{equation}\label{star}
    z_{\alpha}(t)=\int_0^1g(\alpha,\beta)\mathbb Ex_{\beta}(t)\mathrm{d}\beta.
\end{equation}
Combining \eqref{optimalstate}, \eqref{Sigma1} with \eqref{star}, we derive the following GMFG equation system:
\begin{equation}\label{CCR}
  \left\{
  \begin{aligned}
    &\dot{z}_{\alpha}\!=\!\big(A\!-\!BR^{-1}B^{\top}\Pi\big)z_{\alpha}\!+\!D\int_0^1g(\alpha,\beta)z_{\beta}\mathrm{d}\beta\!-\!BR^{-1}B^{\top}\int_0^1g(\alpha,\beta)S_{\beta}\mathrm{d}\beta, \\
    &\dot{S}_{\alpha}\!=\!-\big(A^{\top}-\Pi BR^{-1}B^{\top}\!+\!{2}{\gamma}\Pi\sigma\sigma^{\top}\big)S_{\alpha}\!+\!(Q\Gamma\!-\!\Pi D) z_{\alpha},\\
    &z_{\alpha}(0)\!=\!\int_0^1g(\alpha,\beta)m^x_{\beta}\mathrm{d}\beta,\quad S_{\alpha}(T)\!=\!-Q_f{\Gamma}_fz_{\alpha}(T),\quad \alpha\in [0,1],
  \end{aligned}
  \right.
\end{equation}
which is a family of fully coupled forward-backward equations. In the remaining part of this section, we employ two methods---the fixed point method and the method of continuity---to investigate the well-posedness of the equation system \eqref{CCR}.

\subsection{Fixed point method}
In this subsection, we use a fixed point method to investigate the solvability of \eqref{CCR}.  Let $\Psi_z(t,s)$ and $\Psi_{S}(t,s)$ be the fundamental solution matrices of
\begin{equation}\label{fundamental}
  \dot{y}_1=(A-BR^{-1}B^{\top}\Pi)y_1,\quad \text{and}\quad  \dot{y}_2=-(A^{\top}-\Pi BR^{-1}B^{\top}+{2}{\gamma}\Pi\sigma\sigma^{\top})y_2,
\end{equation}
with $y_1(t), y_2(t)\in \mathbb R^n$ and $\Psi_z(s,s)=\Psi_{S}(s,s)=I_n$. We transform the solvability problem of \eqref{CCR} to a fixed point problem. We write $z_{\alpha}(t)=z(\alpha,t)$ as a function of $(\alpha,t)$ and will derive an equation for $z_{\alpha}$ by eliminating $S_{\alpha}$. Let $C([0,1]\times [0,T]; \mathbb R^n)$ be the function space consisting of continuous $\mathbb R^n$-valued functions defined on $[0,1]\times [0,T]$ with norm $\|x\|=\sup_{\alpha,t}|x(\alpha,t)|$. Define the linear operator $\Xi$ acting on $\tilde z\in C([0,1]\times [0,T]; \mathbb R^n)$:
\begin{equation*}
  \begin{aligned}
    (\Xi\tilde z)(\alpha,t)&=\int_0^t\Psi_z(t,s)D\int_0^1g(\alpha,\beta)\tilde z(\beta, s)\mathrm{d}\beta\mathrm{d}s+\int_0^t\Psi_z(t,s)BR^{-1}B^{\top}\int_0^1g(\alpha,\beta)\\
    &\times \Big\{\int_s^T\Psi_{S}(s,r)(Q\Gamma- \Pi D) \tilde z(\beta, r)\mathrm{d}r
    -\Psi_{S}(s,T)Q_f\Gamma_f \tilde z(\beta,T) \Big\}\mathrm{d}\beta\mathrm{d}s.
  \end{aligned}
\end{equation*}
It is straightforward to show that $\Xi$ is from $C([0,1]\times [0,T]; \mathbb R^n)$ to itself and is continuous. Then the solvability of the GMFG equation system \eqref{CCR} reduces to finding a solution $\tilde z\in C([0,1]\times [0,T];\mathbb R^n)$ to the following fixed point equation:
\begin{equation*}\label{Eq1}
  \tilde z(\alpha,t)=(\Xi\tilde z)(\alpha,t)+\Psi_{z}(t,0)\int_0^1g(\alpha,\beta)m^x_{\beta}\mathrm{d}\beta.
\end{equation*}
Denote the constants $c_g=\max_{\alpha\in [0,1]}\int_0^1g(\alpha,\beta)\mathrm{d}\beta$, $c_{z}=\max_{0\leq t,s\leq T}\big|\Psi_z(t,s)\big|$ and $c_{S}=\max_{0\leq t,s\leq T}\big|\Psi_S(t,s)\big|$. Let $\|\Xi\|$ denote the norm of $\Xi$. We obtain
\begin{equation*}\label{CXi}
  \begin{aligned}
    \|\Xi\|\leq C_{\Xi}:= c_gc_{z}  |D| T +c_gc_{z}c_{S}|BR^{-1}B^{\top}|\cdot (|Q\Gamma- \Pi D|T+|Q_f\Gamma_f|)T.
  \end{aligned}
\end{equation*}
We have the following solvability result.
\begin{theorem}\label{thm2.4}
  Suppose (H1)-(H5) hold. If $C_{\Xi}<1$, then the equation system \eqref{CCR} admits a unique solution $(z_{\cdot}(\cdot), S_{\cdot}(\cdot)) \in C([0,1]\times [0,T]; \mathbb R^n)\times C([0,1]\times [0,T]; \mathbb R^n)$. Furthermore, ODE \eqref{rho} also admits a unique solution $r_{\cdot}( \cdot)\in C([0,1]\times [0,T]; \mathbb R)$.
\end{theorem}
\begin{proof}
    Under the contraction condition, we obtain a unique solution $z$ from \eqref{Eq1}, which in turns determines a unique solution $S_{\alpha}$ for \eqref{Sigma1}, and next $r_{\alpha}$ for \eqref{rho}.
\end{proof}
\begin{remark}
  The contraction condition in the fixed point method may be conservative and typically holds on small time intervals $[0,T]$.
\end{remark}

\subsection{The method of continuity}
This part analyzes solvability by the continuity method under a dominance-monotonicity condition. We need to formulate such a condition for operators acting on an infinite dimensional space. We follow \cite{GCH2023} to introduce some notation.  Define $L^2[0,1]$ as the space of all square integrable functions from the vertex set  $[0,1]$ to $\mathbb R$. Let $(L^2[0,1])^n:=L^2[0,1]\times \cdots\times L^2[0,1]$ with inner product  $\langle {\bf a}, {\bf b}\rangle :=\sum_{i=1}^n\int_0^1{\bf a}_i(\alpha){\bf b}_i(\alpha)\mathrm{d}\alpha=\int_0^1\langle {\bf a}(\alpha), {\bf b}(\alpha)\rangle\mathrm{d}\alpha$, for ${\bf a}, {\bf b}\in (L^2[0,1])^n$, where ${\bf a}_i \in L^2[0,1]$ is the $i$-th component of ${\bf a}$, and ${\bf a}(\alpha)=({\bf a}_1(\alpha),\cdots, {\bf a}_n(\alpha))^{\top}\in \mathbb R^n$  for a fixed index $\alpha\in [0,1]$. The space $(L^2[0,1])^n$ with the above inner produce is a Hilbert space with norm $\|{\bf a}\|_2=(\int_0^1|{\bf a}(\alpha)|^2\mathrm{d}\alpha)^{\frac{1}{2}}$. Let $\mathcal L((L^2[0,1])^n)$ represent the space of bounded linear operators from $(L^2[0,1])^n$ to $(L^2[0,1])^n$ with operator norm $\|\cdot\|_{\rm op}$. Let  $L^2([0,T];(L^2[0,1])^n)$ stand for the Hilbert space of equivalent classes of strongly measurable (in the B\"ochner sense \cite[p. 103]{S1997}) mappings from $[0,T]$ to $(L^2[0,1])^n$ with norm $\|{\bf d}\|_{L^2([0,T];(L^2[0,1])^n)}=(\int_0^T\|{\bf d}(t)\|_2^2\mathrm{d}t)^{\frac{1}{2}}$, and $C([0,T];(L^2[0,1])^n)$ for the Banach space of all continuous mappings from $[0,T]$ to $(L^2[0,1])^n$ with norm
 \begin{equation}\label{Eq2}
     \|{\bf d}\|_C=\sup_{0\leq t\leq T}\|{\bf d}(t)\|_2.
 \end{equation}

A graphon ${\bf G}\in {\bf G}_0^{\rm sp}$ defines a self-adjoint bounded linear operator in $\mathcal L(L^2[0,1])$ as follows (see \cite[p. 124]{L2012}; we use the notation ${\bf G}$ rather than $g$ to emphasize that it is an operator):
\begin{equation}\label{operatorG}
  ({\bf G}{\bf a}_i)(\alpha)=\int_0^1 {\bf G}(\alpha,\beta){\bf a}_i(\beta)\mathrm{d}\beta,\quad \forall \alpha\in [0,1],
\end{equation}
where ${\bf Ga}_i\in L^2[0,1]$. Moreover, graphons can be associated with operators from $(L^2[0,1])^n$ to $(L^2[0,1])^n$. We use the square bracket $[\cdot]$ to indicate that the operator is in $\mathcal L((L^2[0,1])^n)$. For any $F\in \mathbb R^{n\times n}$ and ${\bf G}\in {\bf G}_0^{\rm sp}$, the operator $[F{\bf G}]$ is defined as $([F{\bf G}]{\bf a})(\alpha):= F(({\bf G}{\bf a}_1)(\alpha), \cdots, ({\bf G}{\bf a}_n)(\alpha))^{\top}$, which lies in $\mathbb R^n$ for a given $\alpha$. Since $[F{\bf G}]\in \mathcal L((L^2[0,1])^n)$, it generates a uniformly continuous (hence strongly continuous) semigroup given by $S_{[F{\bf G}]}(t)=\exp(t[F{\bf G}]):=\sum_{j=0}^{\infty}\frac{1}{j!}t^j[F{\bf G}]^j$, $t\geq 0$ \cite[pp. 1-2]{P2012}. Additionally, $\mathbb I\in \mathcal L((L^2[0,1])^n)$ denotes the identity operator satisfying $([F\mathbb I]{\bf a})(\alpha):=F({\bf a}_1(\alpha), \cdots, {\bf a}_n(\alpha))^{\top}=F{\bf a}(\alpha)\in \mathbb R^n.$

Now, we rewrite the equation system \eqref{CCR} in the following operator form:
\begin{equation}\label{CC}
  \left\{
  \begin{aligned}
    &\dot{\bf z}=\big(\mathbb A+[D{\bf G}]\big){\bf z}-[ BR^{-1}B^{\top}{\bf G}]{\bf S},\quad {\bf z}(0)=\int_0^1{\bf G}(\cdot,\beta)m^x_{\beta}\mathrm{d}\beta,\\
    &\dot{\bf S}=-\big(\mathbb A^{\top}+[{2}{\gamma }\Pi\sigma\sigma^{\top}\mathbb I]\big){\bf S}+[(Q\Gamma-\Pi D)\mathbb I]{\bf z}, \quad {\bf S}(T)=-[Q_f\Gamma_f \mathbb I]{\bf z}(T),
  \end{aligned}
  \right.
\end{equation}
where $\mathbb A=[(A-BR^{-1}B^{\top}\Pi)\mathbb I]$. In order to investigate the well-posedness of the forward-backward equation system \eqref{CC}, we introduce the  dominance-monotonicity condition:

\textbf{(H6)} ({Dominance-monotonicity condition}). There exist two constants $\mu\geq 0$, $\nu\geq 0$, a matrix $M\in \mathbb R^{\bar n\times n}$, and two given matrix valued functions $ F, H \in L^{\infty}([0,T];\mathbb R^{\tilde n\times n})$ (for some $\bar n, \tilde n\in \mathbb N$) such that the following conditions are fulfilled:

\noindent (i) One of these two cases holds: Case (A) $\mu>0$ and $\nu=0$; Case (B) $\mu=0$ and $\nu>0$.

\noindent (ii) (Dominance condition). For all $t\in [0,T]$ appearing in $B(t), R(t), D(t), Q(t), \Pi(t)$, $ H(t), F(t)$ and all ${\bf z}$, ${\bf S}\in (L^2[0,1])^n$,
\begin{equation*}\left\{
    \begin{aligned}
        &\| [BR^{-1}B^{\top}{\bf G}] {\bf S}\|_2  \leq \frac{1}{\mu} \|[H\mathbb I]{\bf S}\|_2, \quad \|[(Q\Gamma-\Pi D)\mathbb I]{\bf z}\|_2\leq \frac{1}{\nu}\|[F\mathbb I]{\bf z}\|_2,\\
        &\|[Q_f\Gamma_f\mathbb I]{\bf z}\|_2\leq \frac{1}{\nu}\|[M\mathbb I]{\bf z}\|_2.
    \end{aligned}\right.
\end{equation*}

\noindent (iii) (Monotonicity condition). For all $t\in [0,T]$ appearing in $B(t), R(t), D(t), Q(t)$, $ \Pi(t),  H(t), F(t)$ and all ${\bf z}$, ${\bf S}\in (L^2[0,1])^n$,
\begin{equation*}\left\{
  \begin{aligned}
  &\langle [D{\bf G}]{\bf z}, {\bf S}\rangle -\langle [BR^{-1}B^{\top}{\bf G}]{\bf S}, {\bf S}\rangle -\langle [{2}{\gamma }\Pi\sigma\sigma^{\top}\mathbb I]{\bf S}, {\bf z}\rangle +\langle [(Q\Gamma- \Pi D)\mathbb I]{\bf z}, {\bf z}\rangle\\
  &\qquad\qquad\leq -\nu \|{[F\mathbb I]}{\bf z}\|_2^2-\mu \|{ [H\mathbb I]}{\bf S}\|_2^2,\\
    &-\langle [Q_f\Gamma_f \mathbb I]  {\bf z},  {\bf z} \rangle\geq \nu \|[M\mathbb I] {\bf z}\|_2^2.
  \end{aligned}\right.
\end{equation*}
We adopt the following convention for condition (ii). When $\mu=0$ (resp., $\nu =0$), $\frac{1}{\mu}$ (resp., $\frac{1}{\nu}$) means $\infty$. That is, if $\mu=0$ or $\nu=0$, then the corresponding dominance condition becomes superfluous and is dropped.
\begin{remark}
  In fact, the case with $\mu>0$ and $\nu>0$ is covered by both Case (A) and Case (B).
\end{remark}

For any ${\bf m}, {\bf \bar m}\in L^2([0,T];(L^2[0,1])^n)$ and $\boldsymbol{\eta}\in (L^2[0,1])^n$, we introduce a family of forward-backward equation systems parameterized by $\kappa\in [0,1]$:
\begin{equation}\label{kappa}
  \left\{
  \begin{aligned}
    &\dot{{\bf z}}^{\kappa}\!=\!\kappa \big\{\big(\mathbb A\!+\![D{\bf G}]\big){\bf z}^{\kappa}\!-\![ BR^{-1}B^{\top}{\bf G}]{\bf S}^{\kappa}\big\}\!-\!(1\!-\!\kappa)\mu [{H^{\top}H}\mathbb I]{\bf S}^{\kappa}\!+\!{\bf m},\\
    &\dot{\bf S}^{\kappa}\!=\!\kappa \big\{\!-\!(\mathbb A^{\top}\!+\![{2}{\gamma }\Pi\sigma\sigma^{\top}\mathbb I]){\bf S}^{\kappa}\!+\![(Q\Gamma\!-\!\Pi D)\mathbb I]{\bf z}^{\kappa}\big\}\!-\!(1\!-\!\kappa)\nu[{F^{\top}F}\mathbb I] {\bf z}^{\kappa}\!+\!{\bf\bar m},\\
    &{\bf z}^{\kappa}(0)\!=\!\!\!\int_0^1\!{\bf G}(\cdot,\beta)m^x_{\beta}\mathrm{d}\beta,\!\quad\! {\bf S}^{\kappa}(T)\!=\!-\kappa[Q_f\Gamma_f \mathbb I]{\bf z}^{\kappa}(T)\!+\!(1\!-\!\kappa)\nu [{M^{\top}M}\mathbb I]{\bf z}^{\kappa}(T)\!+\!\boldsymbol{\eta}.
  \end{aligned}
  \right.
\end{equation}
We will show that for each $\kappa\in [0,1]$, the equation system \eqref{kappa} admits a unique solution $({\bf z}^{\kappa},{\bf S}^{\kappa})\in C([0,T]; (L^2[0,1])^n)\times C([0,T];(L^2[0,1])^n)$. It is easy to see that when $\kappa=1$ and $({\bf m}, {\bf\bar m}, \boldsymbol{\eta})=({\bf 0},{\bf 0}, {\bf 0} )$, equation system \eqref{kappa} reduces to \eqref{CC}. Moreover, when $\kappa=0$, \eqref{kappa} becomes
\begin{equation}\label{kappa=0}
  \left\{
  \begin{aligned}
    &\dot{{\bf z}}^{0}= -\mu [{ H^{\top}H}\mathbb I]{\bf S}^{0}+{\bf m},\quad {\bf z}^{0}(0)=\int_0^1{\bf G}(\cdot,\beta)m^x_{\beta}\mathrm{d}\beta,\\
    &\dot{\bf S}^{0}=- \nu [{F^{\top}F}\mathbb I]{\bf z}^{0}+{\bf\bar m}, \quad {\bf S}^{0}(T)=v[{M^{\top}M}\mathbb I]{\bf z}^{0}(T)+\boldsymbol{\eta}.
  \end{aligned}
  \right.
\end{equation}
Now, we investigate the existence and uniqueness of a solution to the forward-backward equation \eqref{kappa=0} in two cases. For Case (A) (i.e., $\mu>0$ and $\nu=0$),  \eqref{kappa=0} is in a decoupled form, for which we can solve the backward equation first to get ${\bf S}^0$, and next obtain ${\bf z}^0$. For Case (B) (i.e., $\mu=0$ and $\nu>0$), \eqref{kappa=0} is also in a decoupled form and easily solved. Thus, under (H6), \eqref{kappa=0} admits a unique solution $({\bf z}^0, {\bf S}^0)\in C([0,T];(L^2[0,1])^n)\times C([0,T];(L^2[0,1])^n)$.

The method of continuity works by showing that whenever \eqref{kappa} admits a unique solution for some $\kappa_0\in [0,1)$, then unique solvability still holds for all $\kappa$ slightly large than $\kappa_0$, allowing to cover up to $\kappa=1$. To develop the continuity arguments, we first give the following a priori estimate of the solution to \eqref{kappa}.

\begin{lemma}\label{lem91221}
  Under (H1)-(H6), suppose that $({\bf z}^{\kappa}_j, {\bf S}_j^{\kappa})\in C([0,T];(L^2[0,1])^n)\times C([0,T]$; $(L^2[0,1])^n)$, $j=1,2$, are solutions of \eqref{kappa} corresponding to $({\bf m}_j, {\bf \bar m}_j, \boldsymbol{\eta}_j)$, respectively. Then one has the following estimate
  \begin{equation*}
    \begin{aligned}
      \|\delta{\bf z}^{\kappa} \|^2_C+\|\delta{\bf S}^{\kappa} \|^2_C\leq K\Big[ \|\delta\boldsymbol{\eta}\|^2_{2}+  \int_0^T\|\delta{\bf m}(t) \|_{2}^2\mathrm{d}t   +  \int_0^T\|\delta{\bf \bar m}(t) \|^2_{2}\mathrm{d}t   \Big],
    \end{aligned}
  \end{equation*}
  where the norm $\|\cdot\|_C$ is defined in \eqref{Eq2}. Here, $\delta {\bf z}^{\kappa}={\bf z}^{\kappa}_1-{\bf z}^{\kappa}_2$ and other terms are similarly defined. $K$ is a positive constant depending on $T$, $\mu$, $\nu$ and the upper bounds of all coefficients in \eqref{state} and \eqref{cost}, $\Pi$, the norm of $H$ $F$ and $M$, and $\|{\bf G}\|_{\rm op}$.
\end{lemma}
\begin{proof}
   We present the proofs for Case (A) and Case (B) separately. For Case (A) (i.e., $\mu>0$ and $\nu=0$), according to \eqref{kappa} and (H6)-(ii), it holds that
  \begin{equation*}
    \begin{aligned}
      \|\delta {\bf z}^{\kappa}\|^2_C&\leq K_1\Big( \int_0^T\|\delta{\bf m}(t) \|^2_{2}\mathrm{d}t   +\int_0^T\|[{H}\mathbb I]\delta{\bf S}^{\kappa}(t)\|_2^2\mathrm{d}t\Big),\\
      \|\delta {\bf S}^{\kappa} \|_C^2
     &\leq K_1\Big( \int_0^T\|\delta{\bf \bar m}(t) \|_{2}^2\mathrm{d}t + \|\delta {\bf z}^{\kappa} \|_C^2+\|\delta\boldsymbol{\eta}\|^2_{2}\Big),
    \end{aligned}
  \end{equation*}
  where $K_1>0$ is a constant. Eliminating $\|\delta {\bf z}^{\kappa}\|_C^2$ in the second inequality, we obtain
  \begin{equation}\label{7}
    \begin{aligned}
      \!\|\delta{\bf z}^{\kappa}\|_C^2\!+\!\|\delta{\bf S}^{\kappa}\|_C^2 \!\leq\! K_2\Big[\|\delta\boldsymbol{\eta}\|^2_{2}\!+\!  \int_0^T\!\!\Big\{\|\delta{\bf m}(t) \|^2_{2}\!+\!\|\delta{\bf \bar m}(t) \|^2_{2}\!+\!\|[{ H}\mathbb I]\delta{\bf S}^{\kappa}(t)\|_2^2\Big\}\mathrm{d}t\Big],
    \end{aligned}
  \end{equation}
  where $K_2>0$ is a constant.   Recalling \eqref{kappa}, we derive the equation of $(\delta{\bf z}^{\kappa}, \delta{\bf S}^{\kappa})$. Taking the differential of $\langle \delta{\bf z}^{\kappa}, \delta{\bf S}^{\kappa}\rangle$ and integrating both sides over $[0,T]$, we then use the monotonicity condition (H6)-(iii) to obtain
    \begin{equation}\label{821}
    \begin{aligned}
      &\langle \delta {\bf z}^{\kappa}(T), \delta{\bf S}^{\kappa}(T)\rangle \!-\!\langle \delta {\bf z}^{\kappa}(0), \delta{\bf S}^{\kappa}(0)\rangle \\
      &\!\!=\!\!\int_0^T\!\!\!\Big\{\big\langle \delta {\bf z}^{\kappa}, \kappa [-(\mathbb A^{\top}\!+\![{2}{\gamma }\Pi\sigma\sigma^{\top}\mathbb I])\delta {\bf S}^{\kappa}\!+\![(Q\Gamma\!-\!D^{\top}\Pi)\mathbb I]\delta {\bf z}^{\kappa}]\!+\!\delta{\bf\bar m}\big\rangle\\
      & \quad \!+\!\big\langle \kappa \big[\big(\mathbb A\!+\![D{\bf G}]\big)\delta{\bf z}^{\kappa}\!-\![ BR^{-1}B^{\top}{\bf G}]\delta{\bf S}^{\kappa}\big]\!-\!(1\!-\!\kappa)\mu [{H^{\top}H}\mathbb I]\delta{\bf S}^{\kappa}\!+\!\delta{\bf m}, \delta{\bf S}^{\kappa}\big\rangle\Big\} \mathrm{d}t\\
      &\!\!\leq \int_0^T\Big(\!-\!\mu \|[{H}\mathbb I]\delta{\bf S}(t)\|_2^2\! +\!\langle \delta {\bf m}(t), \delta {\bf S}^{\kappa}(t)\rangle\! +\!\langle \delta{\bf z}^{\kappa}(t), \delta{\bf\bar m}(t)\rangle\Big) \mathrm{d}t,
    \end{aligned}
  \end{equation}
  and
  \begin{equation}\label{824}\begin{aligned}
      \langle\delta{\bf z}^{\kappa}(T), \delta{\bf S}^{\kappa}(T)\rangle &= \langle \delta{\bf z}^{\kappa}(T), -\kappa[Q_f\Gamma_f \mathbb I]\delta{\bf z}^{\kappa}(T)\!+\!\delta\boldsymbol{\eta}\rangle\geq \langle \delta {\bf z}^{\kappa}(T), \delta \boldsymbol{\eta}\rangle.
  \end{aligned}\end{equation}
  Then we have
  \begin{equation*}
    \begin{aligned}
      &\mu \int_0^T\|[{H}\mathbb I]\delta{\bf S}^{\kappa}(t)\|_2^2\mathrm{d}t\leq \int_0^T\Big(\langle \delta {\bf m}(t), \delta {\bf S}^{\kappa}(t)\rangle +\langle \delta{\bf z}^{\kappa}(t), \delta{\bf\bar m}(t)\rangle\Big) \mathrm{d}t-\langle \delta {\bf z}^{\kappa}(T), \delta \boldsymbol{\eta}\rangle,
    \end{aligned}
  \end{equation*}
  which implies
  \begin{equation}\label{8}
    \begin{aligned}
      &\int_0^T\|[{ H}\mathbb I]\delta{\bf S }^{\kappa}(t)\|_2^2\mathrm{d}t\\
      &\!\leq\!\frac{1}{\mu}\int_0^T\Big(\langle \delta {\bf m}(t), \delta {\bf S}^{\kappa}(t)\rangle \!+\!\langle \delta{\bf z}^{\kappa}(t), \delta{\bf\bar m}(t)\rangle\Big) \mathrm{d}t\!-\!\frac{1}{\mu}\langle \delta {\bf z}^{\kappa}(T), \delta \boldsymbol{\eta}\rangle\\
      &\!\leq\! \frac{1}{\mu} \Big(\|\delta {\bf S}^{\kappa}\|_C\cdot\int_0^T \|\delta {\bf m}(t)\|_2 \mathrm{d}t\! +\!\|\delta{\bf z}^{\kappa}\|_C\cdot \int_0^T\|\delta{\bf\bar m}(t)\|_2  \mathrm{d}t\!+\!\|\delta {\bf z}^{\kappa}\|_C\cdot\|\delta\boldsymbol{\eta}\|_2\Big)\\
      &\!\leq\! \frac{1}{4\mu \theta_1}\Big(\|\delta\boldsymbol{\eta}\|^2_{2}\!+\! T\!\int_0^T\!\Big\{\|\delta{\bf m} (t)\|^2_{2}  \!+\! \|\delta{\bf \bar m}(t) \|^2_{2}\Big\}\mathrm{d}t \Big)  \!+\!\frac{2\theta_1}{\mu}\Big(\|\delta {\bf S}^{\kappa}\|_C^2\!+\!\|\delta{\bf z}^{\kappa}\|_C^2\Big),
    \end{aligned}
  \end{equation}
  where $\theta_1>0$ is arbitrary.
  We now select $\theta_1=\frac{\mu}{4K_2}$, where $K_2$ is from \eqref{7}. Combining \eqref{7} with \eqref{8}, we obtain the desired result.

   For Case (B) (i.e., $\mu=0$ and $\nu>0$), according to \eqref{kappa} and (H6)-(ii), we have
  \begin{equation*}
    \begin{aligned}
      \|\delta {\bf z}^{\kappa}\|^2_C&\leq K_3\Big( \int_0^T\|\delta{\bf m}(t) \|^2_{2}\mathrm{d}t   +\|\delta{\bf S}^{\kappa}\|_C^2\Big),\\
      \|\delta {\bf S}^{\kappa} \|_C^2
     &\leq K_3\Big( \int_0^T\|\delta{\bf \bar m}(t) \|_{2}^2\mathrm{d}t + \int_0^T\|[F\mathbb I]\delta {\bf z}^{\kappa}(t) \|_2^2\mathrm{d}t+\|[M\mathbb I]\delta {\bf z}^{\kappa}(T) \|_2^2+\|\delta\boldsymbol{\eta}\|^2_{2}\Big),
    \end{aligned}
  \end{equation*}
  where $K_3>0$ is a constant. Eliminating $\|\delta {\bf S}^{\kappa}\|_C^2$ in the first inequality, then
  \begin{equation} \label{822}
    \begin{aligned}
      &\|\delta{\bf z}^{\kappa}\|_C^2\!+\!\|\delta{\bf S}^{\kappa}\|_C^2\\
      &\!\leq\! K_4\Big[\|\delta\boldsymbol{\eta}\|^2_{2}\!+\!\!  \int_0^T\!\!\!\Big\{\|\delta{\bf m}(t) \|^2_{2}\!+\!\|\delta{\bf \bar m}(t) \|^2_{2}\!+\!\|[{ F}\mathbb I]\delta{\bf z}^{\kappa}(t)\|_2^2\Big\}\mathrm{d}t\!+\!\|[M\mathbb I]\delta {\bf z}^{\kappa}(T) \|_2^2\Big],
    \end{aligned}
  \end{equation}
  where $K_4>0$ is a constant. Similar to \eqref{821}-\eqref{824}, and using the monotonicity condition (H6)-(iii), we have   
    \begin{equation*}
    \begin{aligned}
      &\langle \delta {\bf z}^{\kappa}(T), \delta{\bf S}^{\kappa}(T)\rangle \!-\!\langle \delta {\bf z}^{\kappa}(0), \delta{\bf S}^{\kappa}(0)\rangle \\
      &\!\!=\!\!\int_0^T\!\!\!\Big\{\big\langle \delta {\bf z}^{\kappa}, \kappa [-(\mathbb A^{\top}\!+\![{2}{\gamma }\Pi\sigma\sigma^{\top}\mathbb I])\delta {\bf S}^{\kappa}\!+\![(Q\Gamma\!-\!D^{\top}\Pi)\mathbb I]\delta {\bf z}^{\kappa}]\!-\!(1\!-\!\kappa)\nu[{F^{\top}F}\mathbb I] \delta{\bf z}^{\kappa}\!+\!\delta{\bf\bar m}\big\rangle\\
      &\quad\quad\!+\!\big\langle \kappa \big[\big(\mathbb A\!+\![D{\bf G}]\big)\delta{\bf z}^{\kappa}\!-\![ BR^{-1}B^{\top}{\bf G}]\delta{\bf S}^{\kappa}\big]\!+\!\delta{\bf m}, \delta{\bf S}^{\kappa}\big\rangle\Big\} \mathrm{d}t\\
      &\!\!\leq \int_0^T\Big(\!-\!\nu\|[{ F}\mathbb I]\delta{\bf z}^{\kappa}(t)\|_2^2\! +\!\langle \delta {\bf m}(t), \delta {\bf S}^{\kappa}(t)\rangle\! +\!\langle \delta{\bf z}^{\kappa}(t), \delta{\bf\bar m}(t)\rangle\Big) \mathrm{d}t,
    \end{aligned}
  \end{equation*}
  and
  \begin{equation*}\begin{aligned}
      \langle\delta{\bf z}^{\kappa}(T), \delta{\bf S}^{\kappa}(T)\rangle &= \langle \delta{\bf z}^{\kappa}(T), -\kappa[Q_f\Gamma_f \mathbb I]\delta{\bf z}^{\kappa}(T)\!+\!(1\!-\!\kappa)\nu [{M^{\top}M}\mathbb I]\delta {\bf z}^{\kappa}(T)\!+\!\delta\boldsymbol{\eta}\rangle \\
      &\geq \nu \|[{M}\mathbb I]\delta{\bf z}^{\kappa}(T)\|_2^2+\langle \delta {\bf z}^{\kappa}(T), \delta \boldsymbol{\eta}\rangle.
  \end{aligned}\end{equation*}
  Then we have
  \begin{equation*}
    \begin{aligned}
      &\nu \|[{M}\mathbb I]\delta{\bf z}^{\kappa}(T)\|_2^2+\nu \int_0^T\|[{F}\mathbb I]\delta{\bf z}^{\kappa}(t)\|_2^2\mathrm{d}t\\
      &\leq \int_0^T\Big(\langle \delta {\bf m}(t), \delta {\bf S}^{\kappa}(t)\rangle +\langle \delta{\bf z}^{\kappa}(t), \delta{\bf\bar m}(t)\rangle\Big) \mathrm{d}t-\langle \delta {\bf z}^{\kappa}(T), \delta \boldsymbol{\eta}\rangle,
    \end{aligned}
  \end{equation*}
  which implies
  \begin{equation} \label{823}
    \begin{aligned}
      &\|[{M}\mathbb I]\delta{\bf z}^{\kappa}(T)\|_2^2+ \int_0^T\|[{F}\mathbb I]\delta{\bf z}^{\kappa}(t)\|_2^2\mathrm{d}t\\
      &\!\leq\!\frac{1}{\nu}\int_0^T\Big(\langle \delta {\bf m}(t), \delta {\bf S}^{\kappa}(t)\rangle \!+\!\langle \delta{\bf z}^{\kappa}(t), \delta{\bf\bar m}(t)\rangle\Big) \mathrm{d}t\!-\!\frac{1}{\nu}\langle \delta {\bf z}^{\kappa}(T), \delta \boldsymbol{\eta}\rangle\\
      &\!\leq\! \frac{1}{\nu} \Big(\|\delta {\bf S}^{\kappa}\|_C\cdot\int_0^T \|\delta {\bf m}(t)\|_2 \mathrm{d}t\! +\!\|\delta{\bf z}^{\kappa}\|_C\cdot \int_0^T\|\delta{\bf\bar m}(t)\|_2  \mathrm{d}t\!+\!\|\delta {\bf z}^{\kappa}\|_C\cdot\|\delta\boldsymbol{\eta}\|_2\Big)\\
      &\!\leq\! \frac{1}{4\nu \theta_2}\Big(\|\delta\boldsymbol{\eta}\|^2_{2}\!+\! T\int_0^T\!\Big\{\|\delta{\bf m} (t)\|^2_{2}   \!+\!   \|\delta{\bf \bar m}(t) \|^2_{2}\Big\}\mathrm{d}t \Big)  \!+\!\frac{2\theta_2}{\nu}\Big(\|\delta {\bf S}^{\kappa}\|_C^2\!+\!\|\delta{\bf z}^{\kappa}\|_C^2\Big),
    \end{aligned}
  \end{equation}
  where $\theta_2>0$ is arbitrary.
  We now select $\theta_2=\frac{\mu}{4K_4}$, where $K_4$ is from \eqref{822}. Combining \eqref{822} with \eqref{823}, we obtain the desired result.
\end{proof}

With the aid of Lemma \ref{lem91221}, we obtain the following lemma.

\begin{lemma}\label{lem2145}
  Under (H1)-(H6), there exists a constant $\delta_0>0$ such that if the equation system \eqref{kappa} is uniquely solvable for some    $\kappa_0\in [0,1)$  whenever $\boldsymbol{\eta}\in (L^2[0,1])^n$ and ${\bf m}$, ${\bf \bar m}\in L^2([0,T];(L^2[0,1])^n)$, then the same conclusion holds for all
  $\kappa \in [\kappa_0,(\kappa_0$ $+\delta_0)\wedge 1]$, where $\delta_0$ does not depend on $\kappa_0$.
\end{lemma}
\begin{proof}
    Take any $\bar\delta\in (0, \delta_0]$, where $\delta_0>0$ will be selected later. For any $({\bf z}, {\bf S})\in C([0,T];(L^2[0,1])^n)\times C([0,T];(L^2[0,1])^n)$, consider the following forward-backward equation system
    \begin{equation}\label{kappa0}
      \left\{
      \begin{aligned}
        &{\bf\dot{\bf\tilde z}}\!=\!\kappa_0 \big[\big(\mathbb A\!+\![D{\bf G}]\big){\bf\tilde z}\!-\![ BR^{-1}B^{\top}{\bf G}]{\bf\tilde S}\big]\!-\!(1\!-\!\kappa_0)\mu[{ H^{\top}H}\mathbb I] {\bf\tilde S}\!+\!{\bf n},\\
        &{\bf \dot{\bf\tilde S}}\!=\!\kappa_0 \big[\!-\!(\mathbb A^{\top}\!+\![{2}{\gamma }\Pi\sigma\sigma^{\top}\mathbb I]){\bf \tilde S}\! +\![(Q\Gamma\!-\!\Pi D)\mathbb I]{\bf \tilde z} \big]\!-\!(1\!-\!\kappa_0)\nu [{F^{\top}F}\mathbb I] {\bf \tilde z}\!+\!{\bf\bar n},\\
    &{\bf\tilde z} (0)\!=\!\int_0^1{\bf G}(\cdot,\beta)m^x_{\beta}\mathrm{d}\beta,\quad {\bf\tilde S} (T)\!=\!-\kappa_0[Q_f\Gamma_f \mathbb I]{\bf\tilde z} (T)\!+\!(1\!-\!\kappa_0)\nu [{M^{\top}M}\mathbb I] {\bf\tilde z}(T)\!+\!\boldsymbol{\bar\eta},
      \end{aligned}
      \right.
    \end{equation}
    where 
    \begin{equation*}
      \begin{aligned}
        &{\bf n}={\bf m}+\bar\delta \big[\big(\mathbb A+[D{\bf G}]\big){\bf  z}-[ BR^{-1}B^{\top}{\bf G}]{\bf S}\big]+\bar\delta\mu[{ H^{\top}H}\mathbb I]  {\bf S},\\
        &{\bf \bar n}={\bf \bar m}+\bar\delta\big [-(\mathbb A^{\top}+[{2}{\gamma }\Pi\sigma\sigma^{\top}\mathbb I]){\bf S}+[(Q\Gamma-\Pi D)\mathbb I]{\bf z}\big]+\bar\delta\nu[{F^{\top}F}\mathbb I] {\bf S},\\
        &\boldsymbol{\bar\eta}=\boldsymbol{\eta}-\bar\delta[Q_f\Gamma_f \mathbb I]{\bf z} (T)-\bar\delta\nu [{M^{\top}M}\mathbb I] {\bf z}(T).
      \end{aligned}
    \end{equation*}
    Here ${\bf m}$, ${\bf\bar m}$ and $\boldsymbol{\eta}$ take the same form as in \eqref{kappa}. We can easily check that $\boldsymbol{\bar\eta}\in (L^2[0,1])^n$ and ${\bf n}$, ${\bf \bar n}\in L^2([0,T];(L^2[0,1])^n)$. Then by our assumption, \eqref{kappa0} admits a unique solution $({\bf\tilde z}, {\bf\tilde S})\in C([0,T];(L^2[0,1])^n)\times C([0,T];(L^2[0,1])^n)$ for some $\kappa_0\in [0,1)$ whenever $\boldsymbol{\eta}\in (L^2[0,1])^n$ and ${\bf m}$, ${\bf \bar m}\in L^2([0,T];(L^2[0,1])^n)$. Thus we have a well-defined mapping from $({\bf z}, {\bf S})$ to $({\bf \tilde z}, {\bf \tilde S})$, which is denoted by
    \begin{equation*}
      ({\bf \tilde z}, {\bf \tilde S})={\bf I}_{\kappa_0+\bar\delta}({\bf z}, {\bf S}).
    \end{equation*}
    Next, we will show that, if $\bar\delta$ is sufficiently small, the operator ${\bf I}_{\kappa_0+\bar\delta}$ from $C([0,T]$; $(L^2[0,1])^n)\times C([0,T];(L^2[0,1])^n)$ to itself is a contraction. Let $({\bf \tilde z_1}, {\bf \tilde S_1})$ be another solution of \eqref{kappa0}, when $({\bf z}, {\bf S})$ is replaced by $({\bf z}_1, {\bf S}_1)$. Denote $\delta{\bf \tilde z}={\bf \tilde z}_1-{\bf \tilde z}$ and $\delta {\bf\tilde S}={\bf \tilde S}_1-{\bf \tilde S}$.  Following the proving method of Lemma \ref{lem91221}, we can show $\|(\delta{\bf \tilde z}, \delta{\bf \tilde S})\|^2_C\leq K_5\bar\delta^2\|(\delta{\bf z}, \delta{\bf S})\|_C^2$, where $K_5$ is a positive constant independent of $\kappa_0$ and $\bar\delta$. Thus we can choose $\delta_0 =\frac{1}{2\sqrt{K_5}}$, so that for any $\bar\delta\in (0,\delta_0]$ satisfying $\kappa_0+\bar\delta\leq 1$, we have $\|(\delta{\bf \tilde z}, \delta{\bf \tilde S})\|^2_C\leq \frac{1}{2}\|(\delta{\bf z}, \delta{\bf S})\|_C^2$, which implies that the operator ${\bf I}_{\kappa_0+\bar\delta}$ is a contraction. Therefore, it has a unique fixed point, which is the unique solution of \eqref{kappa} with $\kappa=\kappa_0+\bar\delta$.
\end{proof}

We have the following existence  result.
\begin{theorem}\label{continuity theorem}
  Under (H1)-(H6), the forward-backward equation system \eqref{CC} admits a unique solution $({\bf z}, {\bf S})\in C([0,T];(L^2[0,1])^n)\times C([0,T];(L^2[0,1])^n)$.
\end{theorem}
\begin{proof}
    From \eqref{kappa=0}, we know that \eqref{kappa} admits a unique solution when $\kappa=\kappa_0=0$. It then follows from Lemma \ref{lem2145} that there exists a positive constant $\delta_0$ such that for each $\kappa \in [\kappa_0,(\kappa_0+\delta_0)\wedge 1]$, \eqref{kappa} admits a unique solution. Thus, we can repeat this process $\iota$ times with $1\leq \iota\delta_0<1+\delta_0$. It then follows that, in particular, the forward-backward equation system \eqref{kappa} admits a unique solution when $\kappa=1$ and $({\bf m}, {\bf \bar m}, \boldsymbol{\eta})=({\bf 0}, {\bf 0}, {\bf 0})$.
\end{proof}

\subsection{Example}
For illustration of the monotonicity condition (H6)-(iii), we give the following example.
\begin{example}\label{exa}
	 Each graphon ${\bf G}$ gives a self-adjoint compact operator from $L^2[0,1]$ to $L^2[0,1]$ \cite[Section 7.5]{L2012}, and hence  has the following discrete spectral decomposition
     \begin{equation}\label{spectral}
  {\bf G}(\alpha,\beta)=\sum_{\ell\in \mathcal I_{\lambda}} \lambda_{\ell}{\bf f}_{\ell}(\alpha){\bf f}_{\ell}(\beta),
\end{equation}
   where ${\bf f}_{\ell}(\cdot)\in L^2[0,1]$,  $\{\lambda_{\ell}\}_{\ell\geq 1}\subset {\mathbb R}$ is the set of eigenvalues, and the convergence is in the $(L^2[0,1])^n$ sense. Here, $\{{\bf f}_{\ell}\}_{\ell\in \mathcal I_{\lambda}}$ represents the corresponding set of orthonormal eigenfunctions,  and  $\mathcal I_{\lambda}$ is the index multiset for all the nonzero eigenvalues of ${\bf G}$. Since the graphon operator ${\bf G}$ defined in \eqref{operatorG} is a Hilbert-Schmidt integral operator and hence a compact operator in $\mathcal L(L^2[0,1])$, the number of elements in $\mathcal I_{\lambda}$ is either finite or countably infinite; see  \cite[Section 7.5]{L2012}. In particular, if $\{\lambda_1,\lambda_2,\cdots\}$ is a countable multiset of nonzero eigenvalues, then $\lim_{\ell\rightarrow \infty}\lambda_{\ell}=0$, and every nonzero eigenvalue has finite multiplicity.

   Now we consider a graphon with  spectral decomposition \eqref{spectral} and denote the index set of all strictly positive eigenvalues $\lambda_{\ell}$ by $\mathcal I_{\lambda}^+$. We further assume that  $\mathcal I_{\lambda}^+$ is finite. Without loss of generality, we assume $\mathcal I_{\lambda}^+=\{\lambda_1,\cdots, \lambda_d\}$ and $\lambda_{\min}=\min_{1\leq \ell\leq d}\lambda_{\ell}$. Let $\mathcal H$ be the subspace generated by the orthonormal basis functions $\{{\bf f}_1,\cdots, {\bf f}_{d}\}$ and $\mathcal H^{\bot}$ be the complement subspace in $L^2[0,1]$. Assume
   \begin{equation}\label{subspace}
       m^x_{\cdot}\in {\mathcal H}^n:=\mathcal H\times \cdots\times \mathcal H.
   \end{equation}
   In view of the proof of Lemma \ref{lem2145}, the existence and uniqueness analysis can be carried out with the solution lying in $\mathcal H^n$ subject to \eqref{subspace}. Then we only need to establish the monotonicity condition for ${\bf z}, {\bf S}\in \mathcal H^n$. Moreover, we assume that for every $t\in [0,T]$, the following inequalities hold:
     \begin{equation}\label{24412}
         \begin{aligned}
             &-\frac{Q\Gamma+\Gamma^{\top} Q-\Pi D-D^{\top}\Pi}{2}-{\gamma^2}\Pi\sigma\sigma^{\top}\sigma^{\top}\sigma\Pi-\frac{DD^{\top}}{4}\geq 0,\\
             &BR^{-1}B^{\top}\lambda_{\min}-2I_{n}\gg 0,\qquad -(Q_f\Gamma_f+\Gamma_f^{\top}Q_f)\geq 0,
         \end{aligned}
     \end{equation}
     where $\Pi$ is determined by \eqref{Phi}. Let $\nu=\inf_{0\leq t\leq T}|\tilde\lambda^*_Q (t)\vee \tilde\lambda^*_{Q_f}|$ and $\mu=\inf_{0\leq t\leq T}|\tilde\lambda_R^*(t)|$, where $\tilde\lambda_Q^*(t)$ (resp., $\tilde\lambda_R^*(t)$, $\tilde\lambda_{Q_f}^*$) is the largest eigenvalue of the symmetric matrix $\frac{Q(t)\Gamma+\Gamma^{\top} Q(t)-D^{\top}(t)\Pi(t)-\Pi(t) D(t)}{2}+{\gamma^2}\Pi(t)\sigma(t)\sigma^{\top}(t)\sigma^{\top}(t)\sigma(t)\Pi(t)+\frac{D(t)D^{\top}(t)}{4}$ (resp., $2I_n$ $-B(t)R^{-1}(t)B^{\top}(t)\lambda_{\min}$, $\frac{Q_f\Gamma_f+\Gamma_f^{\top} Q_f}{2}$). By \eqref{24412}, we obtain $\mu>0$, then
\begin{equation*}
  \begin{aligned}
    &\langle [D{\bf G}]{\bf z}, {\bf S}\rangle -\langle [BR^{-1}B^{\top}{\bf G}]{\bf S}, {\bf S}\rangle -\langle [{2}{\gamma }\Pi\sigma\sigma^{\top}\mathbb I]{\bf S}, {\bf z}\rangle +\langle [(Q\Gamma-\Pi D)\mathbb I]{\bf z}, {\bf z}\rangle\\
    &=\langle [(Q\Gamma-\Pi D)\mathbb I]{\bf z}, {\bf z}\rangle -\langle [BR^{-1}B^{\top}{\bf G}]{\bf S}, {\bf S}\rangle \\
    &\qquad -{2}{\gamma} \Pi\sigma\sigma^{\top}\int_0^1{\bf S}(\alpha) {\bf z}(\alpha)\mathrm{d}\alpha+D\int_0^1\int_0^1{\bf G}(\alpha,\beta){\bf z}(\beta)\mathrm{d}\beta {\bf S}(\alpha)\mathrm{d}\alpha\\
    &\leq \Big\langle \Big[\Big(Q\Gamma\!-\!\Pi D\!+\!{\gamma^2}\Pi\sigma\sigma^{\top}\sigma^{\top}\sigma\Pi+\frac{DD^{\top}}{4}\Big)\mathbb I\Big]{\bf z}, {\bf z}\Big\rangle\!+\!\langle [(2I_n\!-\!BR^{-1}B^{\top}\lambda_{\min})\mathbb I]{\bf S},{\bf S}\rangle \\
    &\leq \tilde\lambda_Q^*(t)\|{\bf z}\|_2^2  \!+\!\tilde\lambda_R^*(t)\|{\bf S}\|_2^2\\
    &\leq -\nu \|{\bf z}\|_2^2-\mu \|{\bf S}\|_2^2,
  \end{aligned}
\end{equation*}
for all ${\bf z}, {\bf S}\in \mathcal H^n$,   and
moreover,
\begin{equation*}
  \begin{aligned}
   -\langle [Q_f\Gamma_f \mathbb I]  {\bf z},  {\bf z} \rangle\geq  -\tilde\lambda_{Q_f}^* \|{\bf z}\|_2^2\geq \nu \| {\bf z}\|_2^2.
  \end{aligned}
\end{equation*}
Thus the desired monotonicity condition is verified.
	\end{example}

\begin{remark}
  i) Example \ref{exa} shows how to construct concrete models to verify the monotonicity condition. In fact, even if $d:=|\mathcal I_{\lambda}^+|=\infty$, as long as $m_{\cdot}^x$  remains in a fixed finite dimensional subspace of the space spanned by $\{{\bf f}_{\ell}, \ell \in \mathcal I_{\lambda}^+\}$, the monotonicity condition can be verified provided that  inequalities \eqref{24412} hold.

  ii) In the risk-neutral case (i.e., the limiting case of \eqref{CCR} with $\gamma\to 0$), the constraints on the initial mean $m_{\cdot}^x$ imposed in Example \ref{exa} are no longer required. For example, we consider the uniform attachment graphs, which converge to the limit graphon ${\bf G}(\alpha,\beta)=1-\max(\alpha,\beta)$, $\alpha,\beta\in [0,1]$, under the cut metric with probability one (see \cite[Proposition 11.40]{L2012}). The spectral decomposition of ${\bf G}$ is given by
  \begin{equation*}
  	{\bf G}(\alpha,\beta)=\sum_{k=1,3,5,\cdots}\frac{4}{k^2\pi^2}\sqrt 2\cos\Big(\frac{\pi k\alpha}{2}\Big)\times \sqrt 2\cos \Big(\frac{\pi k \beta}{2}\Big).
  \end{equation*}
  Here, the eigenvalues of ${\bf G}$ are given by $\lambda_{k}=\frac{4}{k^2\pi^2}$ for positive odd integers $k$, and $\lambda_k=0$ for positive even integers $k$.  Then for every $t\in [0,T]$, if $D=0$, $-(Q\Gamma+\Gamma^{\top}Q)\geq 0$, $-(Q_f\Gamma_f+\Gamma^{\top}_fQ_f)\geq 0$, we can easily check that equation system \eqref{CC} satisfies Case (A) for the monotonicity condition. To show this, let $\lambda_{\min}$ be the smallest  eigenvalue of graphon ${\bf G}$ (here $\lambda_{\min}=0$), $\nu=\inf_{0\leq t\leq T}|\lambda_Q^*(t)\vee \lambda_{Q_f}^*|$, $\mu=\inf_{0\leq t\leq T}|\lambda_R^*(t)|$, where $\lambda_Q^*(t)$ (resp., $\lambda_R^*(t)$,  $\lambda_{Q_f}^*$) is the largest eigenvalue of the symmetric matrix $\frac{Q(t)\Gamma+\Gamma^{\top} Q(t)}{2}$ (resp., $-R^{-1}(t)$, $\frac{Q_f\Gamma_f+\Gamma_f^{\top}Q_f }{2}$).  Subsequently, we can verify $\mu>0$,
  \begin{equation*}
      -\langle [Q_f\Gamma_f \mathbb I]  {\bf z},  {\bf z} \rangle\geq  -\lambda^*_{Q_f} \|{\bf z}\|_2^2\geq \nu \| {\bf z}\|_2^2
  \end{equation*}
  and
  \begin{equation*}
  	\begin{aligned}
  			&\langle [D{\bf G}]{\bf z}, {\bf S}\rangle -\langle [BR^{-1}B^{\top}{\bf G}]{\bf S}, {\bf S}\rangle -\langle [{2}{\gamma }\Pi\sigma\sigma^{\top}\mathbb I]{\bf S}, {\bf z}\rangle +\langle [(Q\Gamma-\Pi D)\mathbb I]{\bf z}, {\bf z}\rangle\\
  			&\leq \langle [Q\Gamma\mathbb I]{\bf z}, {\bf z}\rangle -\langle R^{-1} [\sqrt{\lambda_{\min}}B\mathbb I]{\bf S} , [\sqrt{\lambda_{\min}}B\mathbb I]{\bf S}\rangle \\
            &\leq \lambda_Q^*(t)\|{\bf z}\|_2^2\!+\!\lambda_R^*(t)\|[\sqrt{\lambda_{\min}}B\mathbb I]{\bf S}\|_2^2\\
  			&\leq -\nu \|{\bf z}\|_2^2-\mu \|[\sqrt{\lambda_{\min}}B\mathbb I]{\bf S}\|_2^2.
  		\end{aligned}
  \end{equation*}
  Hence the monotonicity condition is verified.
\end{remark}

\section{$\varepsilon$-Nash equilibrium}\label{sec:Nash}
Suppose that $(z,S)$ has been determined by Theorem \ref{continuity theorem}.  Taking the resulting function $S_\alpha (t)$ for the best response control law \eqref{law}, we construct the decentralized individual strategies in the $N$-player model by appropriately matching the parameter $\alpha$. Specifically,  we construct the set of decentralized strategies
$\hat u=(\hat u_1, \cdots, \hat u_N)$ with
\begin{equation}\label{centralized control law}
    \hat u_i=-R^{-1}B^{\top}\Pi\hat x_i-R^{-1}B^{\top}S_{I_i^*},
\end{equation}
where $I^*_i$ is the midpoint of the subinterval $I_i\in \{I_1, \cdots, I_N\}$ of length $\frac{1}{N}$ (see Section \ref{sec:sub:preliminary}), as the nodal index of agent ${\mathcal A}_i$ in the $N$-player model. The resulting closed-loop dynamics for $\hat x_i$ are governed by
\begin{equation}\label{fstate}
    \mathrm{d}\hat x_i\!=\![(A\!-\!BR^{-1}B^{\top}\Pi)\hat x_i\!+\!D\hat x_i^{(N)}\!-\!BR^{-1}B^{\top}S_{I_i^*}]\mathrm{d}t\!+\!\sigma\mathrm{d}w_i,\quad \hat x_i(0)\!=\!\xi_i,
\end{equation}
where $\hat x_i^{(N)}=\frac{1}{N}\sum_{j=1}^Ng_{ij}^N\hat x_j$.

To analyze the performance of  $\hat u$ in terms of  an $\varepsilon$-Nash equilibrium,  we introduce the following definition.

\begin{definition}\label{Def5.1}
  The set of strategies $\hat  u=(\hat  u_1, \cdots, \hat u_N)$   is called an $\varepsilon$-Nash equilibrium with respect to costs $\mathcal J_i$, $1\le i\le N$, if for $\varepsilon\ge 0$, the $N$ inequalities hold:
  \begin{equation*}
    \mathcal J_i(\hat  u_i, \hat u_{-i})\leq \inf_{u_i\in \mathcal U_c^i  }\mathcal J_{i}(u_i,\hat  u_{-i})+\varepsilon,\quad \text{for all }\ 1\leq i\leq N,
  \end{equation*}
  where $\mathcal U_c^i$ is defined by  \eqref{admissibleset} as the centralized strategy set for agent $\mathcal A_i$.
\end{definition}

We need to introduce the following assumption, which specifies the nature of the approximation error between $g^N$ for the finite graph and the graphon function $g$.

 \textbf{(H7)} The sequence $\{g^N, N\ge 2\}$ and the graphon limit $g$ satisfy
\begin{equation*}
  \lim_{N\rightarrow \infty}\max_{1\leq i\leq N}\sum_{j=1}^N\bigg|\frac{g^N_{ij}}{N}-\int_{\beta\in I_j}g({I^*_i},{\beta})\mathrm{d}\beta\bigg|=0.
\end{equation*}

For further estimate of \eqref{fstate}, we need to introduce an auxiliary process.
By taking the nodal parameter in \eqref{law} as  $\alpha =I_i^*\in [0,1]$, we denote the  best response control law
\begin{equation}\label{decentralized control law}
    \bar u_i=-R^{-1}B^{\top}\Pi \bar x_i-R^{-1}B^{\top}S_{I_i^*},
\end{equation}
and accordingly  introduce the mean-field limit dynamics
\begin{equation}\label{destate}
    \mathrm{d}\bar x_i\!=\![(A\!-\!BR^{-1}B^{\top}\Pi)\bar x_i\!+\!Dz_i\!-\!BR^{-1}B^{\top}S_{I_i^*}]\mathrm{d}t\!+\!\sigma\mathrm{d}w_i,\quad \bar x_i(0)\!=\!\xi_i.
\end{equation}
Here, $z_i=z_{I_i^*}$ is intended to approximate  the weighted state-average $\hat x_i^{(N)}$ at node $i$.

Let us introduce the following notation:
\begin{equation*}
      \begin{aligned}
          & \varepsilon_{1,N}=\max_{1\leq i\leq N}\sum_{j=1}^N\Big|\frac{g^N_{ij}}{N}-\int_{\beta\in I_j}g({I^*_i},{\beta})\mathrm{d}\beta\Big|,\\
          & \varepsilon_{2,N}=\sup_{0\leq \alpha\leq 1} \sup_{0\leq t\leq T}\Big| z_{\alpha}^N(t)-z_{\alpha}(t) \Big|,\quad \varepsilon_{3,N}=\sup_{0\leq \alpha \leq 1}\Big| \mathbb E\xi_{\alpha}^N-m_{\alpha}^x \Big|,
      \end{aligned}
  \end{equation*}
  where $z_{\alpha}^N=\sum_{i=1}^N\mathbbm 1_{I_i}(\alpha)z_i$ and $\xi_{\alpha}^N=\sum_{i=1}^N\mathbbm 1_{I_i}(\alpha)\xi_i$.
  Based on (H1), (H7) and the uniform continuity of $z_{\alpha}(t)$ with respect to $\alpha\in [0,1]$, it follows that $\lim_{N\rightarrow \infty} (\varepsilon_{1,N}+ \varepsilon_{2,N}+\varepsilon_{3,N})=0$. Then we have the following estimates.

\begin{lemma}\label{lem422}
  Under (H1)-(H7), there exists a constant $C_0$ independent of $N$ such that
    \begin{align}
      &\sup_{1\leq i\leq N} \sup_{0\leq t\leq T}|\hat x^{(N)}_i\!-\!z_i| \!\leq\!   \frac{C_0}{N}\Big|\sum_{j=1}^N g_{ij}^N (\xi_j\!-\!\mathbb E\xi_j)\Big|  \!+\! \sup_{0\leq t\leq T}\frac{C_0}{N} \Big|\int_0^t \sum_{j=1}^Ng_{ij}^N\sigma \mathrm{d}w_j\Big| \label{Eq32}\\
      &\qquad \!+\! \frac{C_0}{N^2}\sum_{i=1}^N\Big|\sum_{j=1}^N g_{ij}^N (\xi_j\!-\!\mathbb E\xi_j)\Big|  \!+\! \sup_{0\leq t\leq T}\frac{C_0}{N^2}\sum_{i=1}^N \Big|\int_0^t \sum_{j=1}^Ng_{ij}^N\sigma \mathrm{d}w_j\Big|\!+\!C_0 \varepsilon_N,  \notag\\
      &\sup_{1\leq i\leq N} \sup_{0\leq t\leq T}|\hat x_i-\bar x_i| \!\leq\!   \frac{C_0}{N}\Big|\sum_{j=1}^N g_{ij}^N (\xi_j\!-\!\mathbb E\xi_j)\Big|  \!+\! \sup_{0\leq t\leq T}\frac{C_0}{N} \Big|\int_0^t \sum_{j=1}^Ng_{ij}^N\sigma \mathrm{d}w_j\Big| \label{Eq42} \\
      &\qquad \!+\! \frac{C_0}{N^2}\sum_{i=1}^N\Big|\sum_{j=1}^N g_{ij}^N (\xi_j\!-\!\mathbb E\xi_j)\Big|  \!+\! \sup_{0\leq t\leq T}\frac{C_0}{N^2}\sum_{i=1}^N \Big|\int_0^t \sum_{j=1}^Ng_{ij}^N\sigma \mathrm{d}w_j\Big|\!+\!C_0 \varepsilon_N, \notag
    \end{align}
  where $\hat x_i$ is given by \eqref{fstate} and $\bar x_i$ is given by \eqref{destate}. Here, $\varepsilon_N= \varepsilon_{1,N}+ \varepsilon_{2,N}+\varepsilon_{3,N}$.
\end{lemma}
\begin{proof}
  According to \eqref{CCR} and \eqref{fstate}, it holds that
  \begin{equation}\label{1212}
    \left\{\begin{aligned}
      &\mathrm{d}(\hat x^{(N)}_i-z_i)=\Big[A(\hat x^{(N)}_i-z_i)+\frac{B}{N}\sum_{j=1}^Ng_{ij}^N\hat u_j-B\int_0^1g(I^*_i,\beta) \mathbb E\bar u_{\beta} \mathrm{d}\beta\\
      &\qquad\qquad + \frac{D}{N}\sum_{j=1}^Ng^N_{ij}\hat x^{(N)}_j-D\int_0^1g(I^*_i,\beta)z_{\beta}\mathrm{d}\beta  \Big]\mathrm{d}t+\frac{1}{N}\sum_{j=1}^Ng_{ij}^N\sigma\mathrm{d}w_j,\\
      &\hat x^{(N)}_i(0)-z_i(0)=\frac{1}{N}\sum_{j=1}^Ng^N_{ij}\xi_j- \int_0^1g(I_i^*,\beta)m^x_{\beta}\mathrm{d}\beta .
    \end{aligned}\right.
  \end{equation}
  In the following estimates, we use $C$ to denote a generic constant that does not depend on $N$ and may vary from line to line. Recalling the control laws $\hat u_i$ in \eqref{centralized control law}, $\bar u_i$ in  \eqref{decentralized control law} and $\bar u_{\alpha}$ in \eqref{law}, at time $t$,
  we have
  \begin{equation*}
    \begin{aligned}
    &\!\Big|\frac{B}{N}\sum_{j=1}^Ng_{ij}^N\hat u_j-B\int_0^1g(I^*_i,\beta) \mathbb E\bar u_{\beta} \mathrm{d}\beta\Big| \\
    &\!=\! \Big|\frac{BR^{-1}B^{\top}}{N}\sum_{j=1}^Ng_{ij}^N (\Pi\hat x_j\!+\!S_{I_j^*})-BR^{-1}B^{\top}\int_0^1 g(I_i^*,\beta)(\Pi \mathbb E\bar x_{\beta}+S_{\beta})\mathrm{d}\beta\Big| \\
    &\!\leq \! C |\hat x_i^{(N)}\!-\!z_i|  \!+\!C \Big|\sum_{j=1}^N\Big(\frac{g_{ij}^N}{N}\!-\! \int_{\beta\in I_j}g(I^*_i,\beta)\mathrm{d}\beta\Big)S_{I_j^*}\Big|\!+\!C \Big| \int_0^1g(I_i^*, \beta) (S_{\beta}^N\!-\!S_{\beta})\mathrm{d}\beta\Big| \\
    &\!\leq \! C |\hat x_i^{(N)}\!-\!z_i| +C\varepsilon_N,
    \end{aligned}
\end{equation*}
where $S_{\alpha}^N=\sum_{j=1}^N\mathbbm 1_{I_j}(\alpha)S_{I_j^*}$. Here, the last inequality is due to the uniform boundedness of $S_{I_j^*}$ and $\sup_{0\leq t\leq T}|S_{\alpha}^N$ $-S_{\alpha}|\leq \sup_{0\leq t\leq T}C|z_{\alpha}^N-z_{\alpha}|$ resulting from \eqref{CCR}.  Moreover, we also have
  \begin{equation*}
    \begin{aligned}
    \Big|\frac{1}{N}\sum_{j=1}^Ng^N_{ij}\xi_j\!-\! \int_0^1g(I_i^*,\beta)m^x_{\beta}\mathrm{d}\beta\Big|&\leq \frac{C}{N}\Big|\sum_{j=1}^N g_{ij}^N(\xi_j-\mathbb E\xi_j)\Big| +C\varepsilon_N, \\
      \Big|\frac{D}{N}\sum_{j=1}^Ng^N_{ij}\hat x^{(N)}_j-D\int_0^1g(I^*_i,\beta)z_{\beta}\mathrm{d}\beta\Big|
      &\leq \frac{C}{N}\sum_{j=1}^N\big|\hat x^{(N)}_j-z_j\big|+C\varepsilon_N.
    \end{aligned}
  \end{equation*}
  Hence, by \eqref{1212}, for all $s\in [0,T]$, we obtain
  \begin{equation}\label{1312}\begin{aligned}
      \big|\hat x_i^{(N)}(s)\!-\!z_i(s)\big| &
      \! \leq\! C  \int_0^s\!\!\Big( |\hat x_i^{(N)}\!-\!z_i |\!+\!\frac{1}{N}\sum_{j=1}^N |\hat x_j^{(N)}\!-\!z_j |\Big)\mathrm{d}t\!+\!C \varepsilon_N\\
      &  \!+\!\frac{C}{N}\Big|\sum_{j=1}^N g_{ij}^N(\xi_j-\mathbb E\xi_j)\Big|  \!+\!\frac{C}{N}\Big|\int_0^s \sum_{j=1}^Ng_{ij}^N\sigma \mathrm{d}w_j\Big|.
  \end{aligned}\end{equation}
  Adding up the above $N$  equations and using Gronwall's inequality, we derive
  \begin{equation}\label{1322}\begin{aligned}
      &\frac{1}{N}\sum_{j=1}^N \big|\hat x_j^{(N)}(s)\!-\!z_j(s)\big| \!\leq \!\frac{C}{N^2}\sum_{i=1}^N\Big\{\Big|\sum_{j=1}^N g_{ij}^N(\xi_j-\mathbb E\xi_j)\Big|  \!+\! \Big|\int_0^s \sum_{j=1}^Ng_{ij}^N\sigma \mathrm{d}w_j\Big|\Big\}\!+\!C \varepsilon_N.
  \end{aligned}\end{equation}
  Combining \eqref{1312} with \eqref{1322}, we get \eqref{Eq32}. Similarly, \eqref{Eq42} can be proved.
\end{proof}

For the risk-sensitive cost functionals \eqref{cost} and \eqref{lcost}, we need some exponentiated estimates of the difference between the processes in  \eqref{fstate} and \eqref{destate}. For all the remaining part of this section, we take a sufficiently large $\hat N$ and consider $N\geq \hat N$.
\begin{lemma}\label{lem3389}
  Under (H1)-(H7), for any given deterministic  positive function $\Theta=f(N)$, $N\in \mathbb N$, satisfying
  \begin{equation}\label{Thetacondition89}
    f(N)=o\Big(1/\Big({\frac{1}{N}+\varepsilon_N^2}\Big)\Big),
  \end{equation}
  where $\varepsilon_N$ is given by Lemma \ref{lem422}, we have
  \begin{equation}\label{estimateexp89}
  \begin{aligned}
    \sup_{1\leq i\leq N}\mathbb E\big[e^{\Theta\int_0^T|\hat x_i-\bar x_i|^2\mathrm{d}t}\big]=1+o(1),\quad
    \sup_{1\leq i\leq N}\mathbb E\big[e^{\Theta\int_0^T|\hat x_i^{(N)}-z_i|^2\mathrm{d}t}\big]=1+o(1).
  \end{aligned}
\end{equation}
\end{lemma}
\begin{proof}
  By virtue of Lemma \ref{lem422}, for all $t\in [0,T]$, it holds that
 \begin{equation*}\begin{aligned}
     &|\hat x_i(t)-\bar x_i(t)|^2\!\leq\!  \frac{5C_0^2}{N^2} \Big|\sum_{j=1}^N g_{ij}^N(\xi_j-\mathbb E\xi_j)\Big|^2  \!+\!  \frac{5C_0^2}{N^2} \Big(\!\sup_{0\leq t\leq T}\Big|\!\int_0^t \sum_{j=1}^Ng_{ij}^N\sigma \mathrm{d}w_j\Big|\Big)^2 \\
     &\qquad \!+\! \frac{5C_0^2}{N^4} \Big(\sum_{i=1}^N\Big|\sum_{j=1}^N g_{ij}^N(\xi_j-\mathbb E\xi_j)\Big|\Big)^2  \!+\!  \frac{5C_0^2}{N^4} \Big(\!\sup_{0\leq t\leq T}\sum_{i=1}^N\Big|\!\int_0^t \sum_{j=1}^Ng_{ij}^N\sigma \mathrm{d}w_j\Big|\Big)^2 \!+\!5C_0^2 \varepsilon^2_N.
 \end{aligned}\end{equation*}
 In the following estimates, we use $C$ to denote a generic constant that does not depend on $N$ and may vary from line to line.  Thus, it follows from H\"older's inequality that
 \begin{equation}\label{89511}
     \begin{aligned}
         &\mathbb E\bigg[\exp\Big\{\Theta T\Big(\sup_{0\leq t\leq T}|\hat x_i(t)-\bar x_i(t)|^2\Big)\Big\}\bigg]\!\leq \!\Big(\mathbb E\Big[\exp\Big\{\frac{C\Theta }{N^2} \Big|\sum_{j=1}^N g_{ij}^N (\xi_j-\mathbb E\xi_j) \Big|^2\Big\}\Big]\Big)^{\frac{1}{5}}\\
         & \qquad \times\Big( e^{ C\Theta  \varepsilon^2_N} \Big)^{\frac{1}{5}}   \Big(\mathbb E\Big[\exp\Big\{\frac{C\Theta }{N^2}\Big(\sup_{0\leq t\leq T}\Big|\int_0^t \sum_{j=1}^N g_{ij}^N\sigma \mathrm{d}w_j\Big|\Big)^2\Big\}\Big]\Big)^{\frac{1}{5}}\\
         &\qquad \!\times\!\Big(\mathbb E\Big[\exp\Big\{\frac{C\Theta }{N^4} \Big(\sum_{i=1}^N\Big|\sum_{j=1}^N g_{ij}^N (\xi_j-\mathbb E\xi_j) \Big|\Big)^2\Big\}\Big]\Big)^{\frac{1}{5}}\\
         &\qquad\!\times\!\Big(\mathbb E\Big[\exp\Big\{\frac{C\Theta }{N^4}\Big(\sup_{0\leq t\leq T}\sum_{i=1}^N\Big|\int_0^t \sum_{j=1}^N g_{ij}^N\sigma \mathrm{d}w_j\Big|\Big)^2\Big\}\Big]\Big)^{\frac{1}{5}}.
     \end{aligned}
 \end{equation}

For the first term of the right-hand side of \eqref{89511}, using the independence of $\xi_i$ as stated in (H1), we get
 \begin{equation*}
     \begin{aligned}
         \mathbb E\bigg[ \frac{C\Theta }{N^2} \Big|\!\sum_{j=1}^N\! g_{ij}^N (\xi_j\!-\!\mathbb E\xi_j) \Big|^2\bigg]&\!\leq\!\frac{C\Theta }{N^2} \sum_{j=1}^N \mathbb E|\xi_j\!-\!\mathbb E\xi_j|^2 \!\leq\! C\Theta \Big(\frac{1}{N}\!+\!\varepsilon_N^2\Big)\xrightarrow{N\to\infty} 0 ,
     \end{aligned}
 \end{equation*}
 which implies
 \begin{equation}\label{89513}
     \lim_{N\rightarrow \infty} \frac{C\Theta }{N^2} \Big|\sum_{j=1}^N g_{ij}^N(\xi_j-\mathbb E\xi_j) \Big|^2= 0,\qquad \text{in probability}.
 \end{equation}
 Notice that $\xi_j, 1\leq j\leq N$, are independent random variables taking values in a compact set $S_0^x$. Letting $Y_N^i= \sum_{j=1}^Ng_{ij}^N(\xi_j-\mathbb E\xi_j)   $, then by the integration by parts formula and Hoeffding's inequality, for all $\bar K>1$ and $1\leq i\leq N$,  we have
\begin{equation}\label{89512}
    \begin{aligned}
        &\mathbb E\bigg[e^{\frac{C\Theta }{N^2} |Y_N^i|^2}\mathbbm 1_{\big\{e^{\frac{C\Theta }{N^2} |Y_N^i|^2}\geq \bar K\big\}}\bigg]\\
        &\!\leq \!\int_{\bar K}^{\infty} \mathbb P(e^{\frac{C\Theta }{N^2} |Y_N^i|^2 }> y)\mathrm{d}y \!=\!\int_{\bar K}^{\infty} \mathbb P\Big({|Y_N^i| }>\sqrt{\frac{N^2\ln y}{C\Theta }}\Big)\mathrm{d}y\\
        &\!\leq \!\int_{\bar K}^{\infty} 2 \exp\Big(-\frac{N^2\ln y}{C N\Theta }\Big)\mathrm{d}y \!=\!\int_{\bar K}^{\infty} 2y^{-\frac{N}{C\Theta  }}\mathrm{d}y \!=\!\frac{2\bar K^{1-\frac{N}{C \Theta }}}{\frac{N}{C\Theta }-1},
    \end{aligned}
\end{equation}
which tends to $0$ as $\bar K\rightarrow \infty$, uniformly with respect to $N$, implying that the sequence $\{e^{ \frac{C \Theta}{N^2} |Y_N^i|^2}\}_{N\geq \hat N}$ is uniformly integrable. Combining \eqref{89512} with \eqref{89513}, it follows from Vitali convergence theorem (see \cite[pp. 62-63]{CE2015}) that
 \begin{equation}\label{810515}
     \begin{aligned}
           \lim_{N\rightarrow \infty}\mathbb E\Big[\exp\Big\{\frac{C\Theta }{N^2} \Big|\sum_{j=1}^N g_{ij}^N (\xi_j-\mathbb E\xi_j) \Big|^2\Big\}\Big]=1.
     \end{aligned}
 \end{equation}

  For the second term of  the right-hand side of \eqref{89511}, we have
  \begin{equation}\label{810516}
      \lim_{N\rightarrow \infty}  e^{C\Theta  \varepsilon^2_N} \leq  \lim_{N\rightarrow \infty}  e^{C\Theta  (\frac{1}{N}+\varepsilon^2_N)}=1.
  \end{equation}

  For the third term of  the right-hand side of \eqref{89511}, Since $w_i$, $1\leq i\leq N$, are independent standard Brownian motions, $\int_0^t\sum_{j=1}^Ng_{ij}^N\sigma \mathrm{d}w_j$ has its distribution equal to that of $\sqrt{\sum_{j=1}^N(g_{ij}^N)^2}\int_0^t\sigma \mathrm{d}w$, where $w$ is a standard Brownian motion.   Then, recalling  condition \eqref{Thetacondition89},  it follows from the monotonicity of $e^x$ and Doob's maximal inequality for submartingales that, for sufficiently large $N\geq \hat N$,
  \begin{equation*}
      \begin{aligned}
          \mathbb E\Big[\exp\Big\{\frac{2C\Theta}{N^2}\Big(\sup_{0\leq t\leq T}\Big|\int_0^t\! \sum_{j=1}^Ng_{ij}^N\sigma \mathrm{d}w_j\Big|\Big)^2\Big\}\Big]
          &\!\leq \!\mathbb E\Big[\exp\Big\{\frac{2C\Theta }{N} \sup_{0\leq t\leq T} \Big|\int_0^t\sigma \mathrm{d}w\Big|^2\Big\}\Big]\\
          &\!\leq \!4\mathbb E\Big[\exp\Big\{\frac{2C\Theta }{N}  \Big|\int_0^T\sigma \mathrm{d}w\Big|^2\Big\}\Big]\\
          &\!=\!4\Big[\det\Big(I_n-\frac{4C\Theta }{N}\int_0^T\sigma\sigma^{\top}\mathrm{d}t\Big)\Big]^{-\frac{1}{2}}\\
          &\!=\!4+o(1),
      \end{aligned}
  \end{equation*}
which implies that the sequence $\{e^{\frac{C\Theta }{N^2} (\sup_{0\leq t\leq T} |\int_0^t\! \sum_{j=1}^N g_{ij}^N\sigma \mathrm{d}w_j | )^2}\}_{N\geq \hat N}$ is uniformly integrable. Moreover, by Burkholder-Davis-Gundy inequality, we have
\begin{equation*}
    \begin{aligned}
        \mathbb E\bigg[\frac{C\Theta }{N^2}\Big(\sup_{0\leq t\leq T}\Big|\int_0^t\sum_{j=1}^Ng_{ij}^N\sigma \mathrm{d}w_j\Big|\Big)^2\bigg]\leq \frac{C\Theta }{N}\int_0^T|\sigma|^2\mathrm{d}t\xrightarrow{N\to\infty} 0,
    \end{aligned}
\end{equation*}
which implies that
\begin{equation*}
    \begin{aligned}
        \lim_{N\rightarrow \infty} \frac{C\Theta }{N^2}\Big(\sup_{0\leq t\leq T}\Big|\int_0^t\! \sum_{j=1}^N\sigma \mathrm{d}w_j\Big|\Big)^2 =0, \quad \text{in probability}.
    \end{aligned}
\end{equation*}
Then, using Vitali convergence theorem again, we have
\begin{equation}\label{810517}
    \begin{aligned}
        \lim_{N\rightarrow \infty} \mathbb E\Big[\exp\Big\{\frac{C\Theta }{N^2}\Big(\sup_{0\leq t\leq T}\Big|\int_0^t\! \sum_{j=1}^N g_{ij}^N\sigma \mathrm{d}w_j\Big|\Big)^2\Big\}\Big]=1.
    \end{aligned}
\end{equation}

For the fourth and fifth terms, estimates of the same order can be derived by a similar argument.  Combining \eqref{89511} with \eqref{810515}-\eqref{810517}, it holds that
\begin{equation*}
    \begin{aligned}
    \mathbb E\big[e^{\Theta\int_0^T|\hat x_i-\bar x_i|^2\mathrm{d}t}\big]\leq \mathbb E\bigg[\exp\bigg\{\Theta T  \Big( \sup_{0\leq t\leq T} |\hat x_i-\bar x_i|^2 \Big)\bigg\}\bigg]=1+o(1),
    \end{aligned}
\end{equation*}
uniformly with respect to $i$. Then we obtain the first equation in \eqref{estimateexp89}.  The proof of the second one is similar and omitted here.
\end{proof}

For the processes $(x,y,u)$ on $[0,T]$, define
\begin{equation}\label{Lambda}
    \Lambda_T(x,y,u)\!=\!\int_0^T(\|x(t)\!-\!\Gamma y(t)\|_{Q}^2\!+\!\|u(t)\|^2_{R})\mathrm{d}t \!+\!\|x(T)\!-\!\Gamma_f y(T)\|_{Q_f}^2,
\end{equation}
and we further obtain the following estimate.
\begin{lemma}\label{lem34}
  Under (H1)-(H7), for all deterministic positive function $\Theta=f(N)$, $N\in \mathbb N$,  satisfying
  \begin{equation}\label{Thetacondition2}
    \begin{aligned}
     f(N)=o\bigg({1}/\Big({\sqrt{\frac{1}{N}+\varepsilon_N^2}}\bigg)\Big),
    \end{aligned}
  \end{equation}
   we have
  \begin{equation}\label{OTheta}
    \begin{aligned}
      &\sup_{1\leq i\leq N}\mathbb E\bigg[\exp\Big\{\Theta \big|\Lambda_T(\hat x_i, \hat x^{(N)}_i, \hat u_i)-\Lambda_T(\bar x_i, z_i, \bar u_i)\big|\Big\}\bigg]=1+o(1).
    \end{aligned}
  \end{equation}
\end{lemma}
\begin{proof}
    Recalling \eqref{Lambda}, we have
  \begin{equation*}
    \begin{aligned}
      &\big|\Lambda_T(\hat x_i, \hat x^{(N)}_i, \hat u_i)-\Lambda_T(\bar x_i, z_i, \bar u_i)\big|\\
      &\leq \int_0^T \Big\{\big|\|\hat x_i-\Gamma\hat x^{(N)}_i\|_Q^2 -\|\bar x_i-\Gamma z_i\|_Q^2\big|+\big|\|\hat u_i\|^2_R-\|\bar u_i\|_R^2\big|\Big\}\mathrm{d}t\\
      &\qquad\quad +\big|\|\hat x_i(T)-\Gamma_f\hat x^{(N)}_i(T)\|_{Q_f}^2 -\|\bar x_i(T)-\Gamma_f z_i(T)\|_{Q_f}^2\big|.
    \end{aligned}
  \end{equation*}
  Using  H\"older's inequality, we get
  \begin{equation}\label{23510}
      \begin{aligned}
          &\mathbb E\Big[\exp\Big\{\Theta\big|\Lambda_T(\hat x_i, \hat x_i^{(N)},\hat u_i)-\Lambda_T(\bar x_i,z_i,\bar u_i)\big|\Big\}\Big]\\
          &\leq \Big(\mathbb E\Big[e^{3\Theta\int_0^T  \big|\|\hat x_i-\Gamma\hat x^{(N)}_i\|_Q^2 -\|\bar x_i-\Gamma z_i\|_Q^2 \big|\mathrm{d}t}\Big]\Big)^{\frac{1}{3}}\Big(\mathbb E\Big[e^{3\Theta\int_0^T  \big|\|\hat u_i\|^2_R-\|\bar u_i\|_R^2 \big|\mathrm{d}t}\Big]\Big)^{\frac{1}{3}}\\
          &\qquad\quad \times \Big(\mathbb E\Big[e^{3\Theta\big|\|\hat x_i(T)-\Gamma_f\hat x^{(N)}_i(T)\|_{Q_f}^2 -\|\bar x_i(T)-\Gamma_f z_i(T)\|_{Q_f}^2\big|}\Big]\Big)^{\frac{1}{3}}.
      \end{aligned}
  \end{equation}
  Noticing $\langle Q x,x\rangle -\langle Q y,y\rangle =\langle Q(x-y), x-y\rangle +2\langle Q(x-y), y\rangle $ and using  H\"older's inequality, we have
  \begin{equation}\label{1234}
    \begin{aligned}
      &\mathbb E\big[e^{3\Theta \int_0^T |\|\hat x_i-\Gamma\hat x_i^{(N)}\|^2_Q-\|\bar x_i-\Gamma z_i\|^2_Q |\mathrm{d}t}\big]\\
      &\!\leq\! \mathbb E\big[e^{6\Theta \int_0^T \{\|\hat x_i-\bar x_i\|_{Q}^2+\|\hat x_i^{(N)}-z_i\|_{\Gamma^{\top}\! Q\Gamma}^2  + |\langle Q[\hat x_i-\bar x_i -\Gamma( \hat x_i^{(N)}-z_i) ],\bar x_i-\Gamma z_i\rangle |  \}\mathrm{d}t}\big]\\
      &\!\leq\! \Big\{\mathbb E\Big[e^{ C\Theta \int_0^T |\hat x_i-\bar x_i|^2\mathrm{d}t}\Big]\Big\}^{\frac{1}{4}} \Big\{\mathbb E\Big[e^{ C\Theta \int_0^T \big|\langle \hat x_i-\bar x_i ,\bar x_i-\Gamma z_i\rangle\big|\mathrm{d}t}\Big]\Big\}^{\frac{1}{4}}\\
      & \qquad\quad\times \Big\{\mathbb E\Big[e^{ C\Theta \int_0^T|\hat x_i^{(N)}-z_i|^2\mathrm{d}t}\Big]\Big\}^{\frac{1}{4}}\Big\{\mathbb E\Big[e^{ C\Theta \int_0^T\big|
      \langle \hat x_i^{(N)}\!-\!z_i ,\bar x_i-\Gamma z_i\rangle\big| \mathrm{d}t}\Big]\Big\}^{\frac{1}{4}},
    \end{aligned}
  \end{equation}
  where $C>0$ is  a generic constant that does not depend on $N$ and may vary from line to line. By virtue of Lemma \ref{lem3389}, we know that
  \begin{equation}\label{first}
    \begin{aligned}
      &\Big\{\mathbb E\Big[e^{ C\Theta \int_0^T |\hat x_i-\bar x_i|^2\mathrm{d}t}\Big]\Big\}^{\frac{1}{4}}\Big\{\mathbb E\Big[e^{ C\Theta \int_0^T|\hat x_i^{(N)}-z_i|^2\mathrm{d}t}\Big]\Big\}^{\frac{1}{4}} =1+o(1).
    \end{aligned}
  \end{equation}

  Next, we analyze the other two factors in \eqref{1234}. Recalling the fundamental solution matrices $\Psi_z(t,s)$, we get
  \begin{equation*}
    \begin{aligned}
     \bar x_i(t)=\Psi_z(t,0)\xi_i-\int_0^t\Psi_z(t,s)(BR^{-1}B^{\top}S_{I_i^*}-Dz_i)\mathrm{d}s+\int_0^t\Psi_z(t,s)\sigma\mathrm{d}w_i.
    \end{aligned}
  \end{equation*}
 By the GMFG equation system \eqref{CCR} and Theorem \ref{thm2.4}, we know that $(z_i,S_{I_i^*})\in C([0,T];\mathbb R^n)\times C([0,T];\mathbb R^n)$, which implies that $(z_i,S_{I_i^*})$ is uniformly bounded with respect to all $i$. Since the initial state $\xi_i$ takes value in a compact set $S_0^x$ and $\Psi_z(t,s)=\Psi_z(t,0)\Psi_z(0,s)$, we have
 \begin{equation*}
     \begin{aligned}
         \sup_{0\leq t\leq T}|\bar x_i-\Gamma z_i|\leq C+\sup_{0\leq t\leq T}C\Big|\int_0^t\Psi_z(0,s)\sigma \mathrm{d}w_i\Big|.
     \end{aligned}
 \end{equation*}
 Using H\"older's inequality and a method similar to the proof of Lemma \ref{lem3389}, taking $\tilde\epsilon=\sqrt{\frac{1}{N}+\varepsilon_N^2}$, then we obtain
  {\allowdisplaybreaks
      \begin{align}
          &\!\mathbb E\bigg[\exp\Big( C \Theta T\cdot \sup_{0\leq t\leq T}|\langle \hat x_i\!-\!\bar x_i, \bar x_i\!-\!\Gamma z_i\rangle |\Big)\bigg] \notag\\
          &\!\leq\! \mathbb E\Big[\exp\Big(\frac{ C_2\Theta T}{2\tilde\epsilon}\cdot \sup_{0\leq t\leq T}\big|\hat x_i\!-\!\bar x\big|^2\!+\!\frac{ \tilde\epsilon C\Theta T }{2}\cdot\sup_{0\leq t\leq T}\big|\bar x_i\!-\!\Gamma z_i\big|^2\Big)\Big]\notag\\
          &\!\leq\! \Big\{\mathbb E\Big[\exp\Big(\frac{ C\Theta T}{\tilde\epsilon}\cdot \sup_{0\leq t\leq T}\big|\hat x_i\!-\!\bar x_i\big|^2\Big)\Big]\Big\}^{\frac{1}{2}} \Big\{\mathbb E\Big[\exp\Big( \tilde\epsilon C\Theta T\cdot \sup_{0\leq t\leq T}\big|\bar x_i\!-\!\Gamma z_i\big|^2\Big)\Big]\Big\}^{\frac{1}{2}}\notag\\
          &\!\leq\!  [1+o(1)] \Big\{\mathbb E\Big[\exp\Big\{ \tilde\epsilon C\Theta T \Big[C^2+\Big(\sup_{0\leq t\leq T}C\Big|\int_0^t\Psi_z(0,s)\sigma \mathrm{d}w_i\Big|\Big)^2\Big]\Big\}\Big]\Big\}^{\frac{1}{2}}\notag\\
          &\!\leq\!  [1+o(1)] e^{\tilde\epsilon C \Theta T} \Big\{\mathbb E\Big[\exp\Big\{ \tilde\epsilon C \Theta T  \Big(\sup_{0\leq t\leq T}\Big|\int_0^t\Psi_z(0,s)\sigma \mathrm{d}w_i\Big|\Big)^2 \Big\}\Big]\Big\}^{\frac{1}{4}}\notag\\
          &\!=\! 1+o(1),\notag
      \end{align}}\noindent
      where the second to last line is due to Vitali convergence theorem and
      \begin{equation*}
          \begin{aligned}
              &\lim_{N\rightarrow \infty} \tilde\epsilon C \Theta T  \Big(\sup_{0\leq t\leq T}\Big|\int_0^t\Psi_z(0,s)\sigma \mathrm{d}w_i\Big|\Big)^2 =0,\quad \text{in probability}, \\
              &\sup_{N\geq \hat N}\mathbb E\Big[\exp\Big\{2 \tilde\epsilon C \Theta T  \Big(\sup_{0\leq t\leq T}\Big|\int_0^t\Psi_z(0,s)\sigma \mathrm{d}w_i\Big|\Big)^2 \Big\}\Big]\leq 4+o(1).
          \end{aligned}
      \end{equation*}
      Similarly, we also have
      \begin{equation*}
          \begin{aligned}
              \mathbb E\bigg[\exp\Big( C\Theta T \cdot \sup_{0\leq t\leq T} |\langle \hat x_i^{(N)}-z_i ,\bar x_i-\Gamma z_i\rangle| \Big)\bigg]=1+o(1).
          \end{aligned}
      \end{equation*}
Making use of the above estimates, we further obtain
  \begin{equation}\label{second}
    \begin{aligned}
      &\left\{\mathbb E\left[e^{ C\Theta \int_0^T |\langle \hat x_i-\bar x_i ,\bar x_i-\Gamma z_i\rangle|\mathrm{d}t}\right]\right\}^{\frac{1}{4}}   \left\{\mathbb E\left[e^{ C\Theta \int_0^T
      |\langle \hat x_i^{(N)}-z_i ,\bar x_i-\Gamma z_i\rangle| \mathrm{d}t}\right]\right\}^{\frac{1}{4}}\\
      &\!\leq \!\Big\{\mathbb E\Big[\exp\Big\{C\Theta T\Big(\sup_{0\leq t\leq T}\big|\langle \hat x_i-\bar x_i, \bar x_i-\Gamma z_i \rangle\big|\Big)\Big\}\Big]\Big\}^{\frac{1}{4}}\\
      &\qquad \times \Big\{\mathbb E\Big[\exp\Big\{C\Theta T\Big(\sup_{0\leq t\leq T}\big|\langle \hat x_i^{(N)}-z_i, \bar x_i-\Gamma z_i \rangle\big|\Big)\Big\}\Big]\Big\}^{\frac{1}{4}}\\
      &=1+o(1).
    \end{aligned}
  \end{equation}
  Combining \eqref{first} with \eqref{second}, we have
  \begin{equation}\label{mid}
    \begin{aligned}
      &\mathbb E\big[e^{3\Theta \int_0^T |\|\hat x_i-\Gamma\hat x_i^{(N)}\|^2_Q-\|\bar x_i-\Gamma z_i\|^2_Q |\mathrm{d}t}\big]=1+o(1).
    \end{aligned}
  \end{equation}
  Applying a similar method as used in the proof of \eqref{mid}, we also obtain
   \begin{equation}  \label{control term}
      \begin{aligned}
          \mathbb E\big[e^{3\Theta\int_0^T |\|\hat u_i\|^2_R-\|\bar u_i\|_R^2 |\mathrm{d}t}\big]&=1+o(1), \\
          \mathbb E\Big[e^{3\Theta\big|\|\hat x_i(T)-\Gamma_f\hat x^{(N)}_i(T)\|_{Q_f}^2 -\|\bar x_i(T)-\Gamma_f z_i(T)\|_{Q_f}^2\big| }\Big]&=1+o(1).
      \end{aligned}
      \end{equation}
  Combining \eqref{23510} with \eqref{mid}-\eqref{control term}, we  obtain \eqref{OTheta}.
\end{proof}

Suppose that agent $\mathcal A_i$ chooses an alternative feedback control law $u_i\in \mathcal U_c^i$, while the other agents $\mathcal A_j$, $j\ne i$, still adopt the strategies $\hat u_j$ in \eqref{centralized control law}. The resulting closed-loop state processes are given by
\begin{equation}\label{pstate}\left\{\begin{aligned}
  &\mathrm{d} x_i\!=\!(A x_i\!+\!B u_i\!+\!D x^{(N)}_i)\mathrm{d}t\!+\!\sigma \mathrm{d}w_i,\! \quad \!&&x_i(0)\!=\!\xi_i,\\
  &\mathrm{d} x_j\!=\![Ax_j\!-\!BR^{-1}B^{\top}(\Pi x_j\!+\!S_{I_j^*}) \!+\!D x^{(N)}_j]\mathrm{d}t\!+\!\sigma \mathrm{d}w_j, \!\quad\! &&x_j(0)\!=\!\xi_j,\!\quad \!j\ne i,
\end{aligned}\right.\end{equation}
where $x^{(N)}_i=\frac{1}{N}\sum_{l=1}^Ng^N_{il} x_l$. The cost functional of agent $\mathcal A_i$ is
\begin{equation*}
  \begin{aligned}
    &\mathcal J_i( u_i,\hat u_{-i})\!
    =\!\mathbb E\Big[\!\exp\Big({\gamma}\Big\{\!\int_0^T\!\!\big[\|x_i\!-\!\Gamma x^{(N)}_i\|_{Q}^2 \!+\!\|u_i\|^2_{R}\big]\mathrm{d}t \!+\!\| x_i(T)\!-\!\Gamma_f x^{(N)}_i(T)\|_{Q_f}^2\Big\}\Big)\Big].
  \end{aligned}
\end{equation*}

Let $u_{i}^{cl}(t,\omega)$, as a stochastic process adapted to $\mathcal F_t$, be the closed-loop control inputs generated by the above control law $u_i$. Based on \eqref{lstate}, we introduce the auxiliary  mean-field limit dynamics for each agent:
\begin{equation}\label{perturstate}\left\{
  \begin{aligned}
    &\mathrm{d} x_i^{\infty}=(A x_i^{\infty}+Bu_i^{cl}+Dz_i)\mathrm{d}t+\sigma\mathrm{d}w_i,\quad  x_i^{\infty}(0)=\xi_i,\\
    &\mathrm{d} x_j^{\infty}=[Ax_j^{\infty}-BR^{-1}B^{\top}(\Pi x_j^{\infty}+S_{I_j^*})+Dz_j]\mathrm{d}t+\sigma\mathrm{d}w_j,\!\quad \!x_j^{\infty}(0)=\xi_j,\!\quad\! j\ne i,
  \end{aligned}\right.
\end{equation}
where $z_i=z_{I_i^*}$ is specified by \eqref{CCR} and $u_i^{cl}$ is the closed-loop control inputs generated by the control law $u_i$
 in \eqref{pstate}. The use of the form $u_i^{cl}$ is for convenience of further performance estimate (see the subsequent proof of  Theorem \ref{thmepsilon}).
 Then, by  the method of proving Lemma \ref{lem422}, we obtain the following estimates.
\begin{lemma} \label{lemma56}
  Under (H1)-(H7), for \eqref{pstate} and \eqref{perturstate}, there exists a constant $\bar C$ such that
        \begin{align}
      &\sup_{1\leq i\leq N} \sup_{0\leq t\leq T}| x^{(N)}_i-z_i| \!\leq \!\bar C\Big(\frac{1}{N}\!+\!\varepsilon_N\Big)\!+\!\frac{\bar C}{N}\Big|\sum_{j=1}^N g_{ij}^N(\xi_j-\mathbb E\xi_j)\Big|\label{estimate12}\\
      &\qquad\quad\!+\!\frac{\bar C}{N}\int_0^T|u_i^{cl}|\mathrm{d}t\!+\! \sup_{0\leq t\leq T} \frac{\bar C}{N}\Big|\int_0^t \sigma\mathrm{d}w_i\Big|\!+\!\sup_{0\leq t\leq T} \frac{\bar C}{N}\Big|\int_0^t\sum_{j=1}^Ng_{ij}^N\sigma\mathrm{d}w_j\Big|\notag\\
      &\qquad\quad \!+\!\frac{\bar C}{N^2}\sum_{i=1}^N\Big|\sum_{j=1}^N g_{ij}^N(\xi_j-\mathbb E\xi_j)\Big|\!+\!\sup_{0\leq t\leq T} \frac{\bar C}{N^2}\sum_{i=1}^N\Big|\int_0^t\sum_{j=1}^Ng_{ij}^N\sigma\mathrm{d}w_j\Big|, \notag\\
            &\sup_{1\leq i\leq N} \sup_{0\leq t\leq T}| x_i- x^{\infty}_i| \!\leq \!\bar C\Big(\frac{1}{N}\!+\!\varepsilon_N\Big)\!+\!\sup_{1\leq i\leq N}\frac{\bar C}{N}\Big|\sum_{j=1}^N g_{ij}^N (\xi_j-\mathbb E\xi_j)\Big|\label{estimate22}\\
      &\qquad\quad\!+\!\frac{\bar C}{N}\int_0^T|u_i^{cl}|\mathrm{d}t\!+\! \sup_{0\leq t\leq T}\frac{\bar C}{N} \Big|\int_0^t \sigma\mathrm{d}w_i\Big|\!+\! \sup_{0\leq t\leq T} \frac{\bar C}{N}\Big|\int_0^t\sum_{j=1}^N\sigma\mathrm{d}w_j\Big|\notag\\
      &\qquad\quad \!+\!\frac{\bar C}{N^2}\sum_{i=1}^N\Big|\sum_{j=1}^N g_{ij}^N(\xi_j-\mathbb E\xi_j)\Big|\!+\!\sup_{0\leq t\leq T} \frac{\bar C}{N^2}\sum_{i=1}^N\Big|\int_0^t\sum_{j=1}^Ng_{ij}^N\sigma\mathrm{d}w_j\Big|. \notag
        \end{align}
\end{lemma}
\begin{proof}
    By \eqref{CCR} and \eqref{pstate}, it holds that
    \begin{equation}\label{7185282}
        \left\{
        \begin{aligned}
            &\!\mathrm{d}(x_i^{(N)}\!-\!z_i)\!=\!\Big[A(x_i^{(N)}\!-\!z_i)\!+\!\frac{B}{N}g_{ii}^Nu_i^{cl}\!-\!\frac{BR^{-1}B^{\top}}{N}\!\!\sum_{j\ne i}g_{ij}^N(\Pi x_j\!+\!S_{I_j^*})\!+\!\frac{D}{N}\sum_{j=1}^Ng_{ij}^Nx_j^{(N)}\\
            &\!\qquad\quad \!-\!B\int_0^1g(I_i^*,\beta) \mathbb E\bar u_{\beta}\mathrm{d}\beta\!-\!D \int_0^1g(I_i^*,\beta)z_{\beta}\mathrm{d}\beta  \Big]\mathrm{d}t\!+\!\frac{1}{N}\sum_{j=1}^N g_{ij}^N\sigma\mathrm{d}w_j,\\
            &\!x_i^{(N)}(0)\!-\!z_i(0)\!=\!\frac{1}{N}\sum_{j=1}^Ng_{ij}^N\xi_j\!-\! \int_0^1g(I_i^*,\beta)m_{\beta}^x\mathrm{d}\beta.
        \end{aligned}
        \right.
    \end{equation}
    In the following estimates, we again use $C$ to denote a generic constant that does not depend on $N$ and may vary from line to line. Recalling the control laws $\bar u_i$ in \eqref{decentralized control law} and $\bar u_{\alpha}$ in \eqref{law}, we have
    {\allowdisplaybreaks
        \begin{align}
        &\int_0^s\Big|\frac{B}{N}g_{ii}^Nu_i^{cl}\!-\!\frac{BR^{-1}B^{\top}}{N}\sum_{j\ne i}g_{ij}^N(\Pi x_j\!+\!S_{I_j^*})\!-\!B\int_0^1g(I_i^*,\beta) \mathbb E\bar u_{\beta}\mathrm{d}\beta\Big|\mathrm{d}t \notag\\
        &\!\leq\!\int_0^s\Big|\frac{B}{N}g_{ii}^Nu_i^{cl}\!+\!\frac{BR^{-1}B^{\top}}{N}g_{ii}^N(\Pi x_i\!+\!S_{I_i^*})\Big|\mathrm{d}t \notag\\
            &\qquad\!+\!\int_0^s\Big| \frac{BR^{-1}B^{\top}}{N}\sum_{j=1}^Ng_{ij}^N(\Pi x_j\!+\!S_{I_j^*})\!-\!BR^{-1}B^{\top} \int_0^1g(I_i^*,\beta)(\Pi\mathbb E\bar x_{\beta}\!+\!S_{\beta})\mathrm{d}\beta \Big|\mathrm{d}t \notag\\
            &\!\leq\! \int_0^s\Big|\frac{B}{N}g_{ii}^Nu_i^{cl}\!+\!\frac{BR^{-1}B^{\top}}{N}g_{ii}^N(\Pi x_i\!+\!S_{I_i^*})\Big|\mathrm{d}t +C \int_0^s|x^{(N)}_i-z_i|\mathrm{d}t +C\varepsilon_N^2 \notag\\
           &\!\leq \!\frac{C}{N}\int_0^s|u_i^{cl}|\mathrm{d}t+C \int_0^s|x^{(N)}_i-z_i|\mathrm{d}t\!+\!\frac{C}{N}\sup_{0\leq t\leq s} \Big|\int_0^t\sum_{j=1}^Ng_{ij}^N\sigma\mathrm{d}w_j\Big|\notag\\
           &\qquad\quad\!+\!\frac{C}{N}\sup_{0\leq t\leq s} \Big|\int_0^t  \sigma\mathrm{d}w_i\Big| \!+\!C\Big(\frac{1}{N}+\varepsilon_N\Big) , \notag
        \end{align}
    }\noindent
    where the last inequality is due to $|x_i(s)|\leq C+C \int_0^s|u_i^{cl}|\mathrm{d}t+C|\int_0^s\sum_{j=1}^Ng_{ij}^N\sigma\mathrm{d}w_j|+C|\int_0^s\sigma\mathrm{d}w_i|$.      Moreover, we also have
    {\allowdisplaybreaks
        \begin{align}
        \Big|\frac{1}{N}\sum_{j=1}^Ng^N_{ij}\xi_j\!-\! \int_0^1g(I_i^*,\beta)m^x_{\beta}\mathrm{d}\beta\Big|&\leq \frac{C}{N}\Big|\sum_{j=1}^N g_{ij}^N (\xi_j-\mathbb E\xi_j)\Big| +C\varepsilon_N,\notag \\
            \Big|\frac{D}{N}\sum_{j=1}^Ng_{ij}^Nx_j^{(N)}\!-\!D \int_0^1g(I_i^*,\beta)z_{\beta}\mathrm{d}\beta\Big|&\!\leq \!\frac{C}{N}\sum_{j=1}^N\big|x_j^{(N)}\!-\!z_j\big|\!+\!C\varepsilon_N. \notag
        \end{align}
}
    Thus, by \eqref{7185282}, for all $s\in [0,T]$, we have
    \begin{equation}\label{7185312}
        \begin{aligned}
             &\big|x_i^{(N)}(s)\!-\!z_i(s)\big|\\
             &\!\leq \!C \int_0^s\Big(\big|x_i^{(N)}\!-\!z_i\big| \!+\!\frac{1}{N}\sum_{j=1}^N\big|x_j^{(N)}\!-\!z_j\big| \Big)\mathrm{d}t \!+\!C\Big(\frac{1}{N}\!+\!\varepsilon_N\Big)\!+\!\frac{C}{N}\int_0^s|u_i^{cl}|\mathrm{d}t\\
            &\!+\!\frac{C}{N}\Big|\sum_{j=1}^N g_{ij}^N (\xi_j-\mathbb E\xi_j)\Big| \!+\!\frac{C}{N}\sup_{0\leq t\leq s} \Big|\int_0^t\sum_{j=1}^Ng_{ij}^N\sigma\mathrm{d}w_j\Big|\!+\!\frac{C}{N}\sup_{0\leq t\leq s} \Big|\int_0^t \sigma\mathrm{d}w_i\Big|.
        \end{aligned}
    \end{equation}
    Adding up the above $N$ equations and using Gronwall's inequality, we have
    \begin{equation}\label{7185322}\begin{aligned}
        \frac{1}{N}\sum_{j=1}^N \big|x_j^{(N)}(s)\!-\!z_j(s)\big|&\!\leq \!C\Big(\frac{1}{N}\!+\!\varepsilon_N\Big)\!+\!\frac{C}{N^2}\sum_{i=1}^N\Big|\sum_{j=1}^N g_{ij}^N (\xi_j-\mathbb E\xi_j)\Big|\!+\!\frac{C}{N}\int_0^s|u_i^{cl}|\mathrm{d}t\\
        & \!+\!\frac{C}{N^2}\sup_{0\leq t\leq s} \sum_{i=1}^N\Big|\int_0^t\sum_{j=1}^Ng_{ij}^N\sigma\mathrm{d}w_j\Big|\!+\!\frac{C}{N}\sup_{0\leq t\leq s} \Big|\int_0^t \sigma\mathrm{d}w_i\Big|.
        \end{aligned}
    \end{equation}
    Combining \eqref{7185312} with \eqref{7185322}, we get \eqref{estimate12}. Similarly, \eqref{estimate22} can be proved.
\end{proof}

Based on the above analysis, we now present the key result on performance analysis.
\begin{theorem}\label{thmepsilon}
  Suppose (H1)-(H7) hold. The set of strategies $\hat u=(\hat u_1, \cdots, \hat u_N)$  given by \eqref{centralized control law} is an $\varepsilon$-Nash equilibrium,
  where $\varepsilon =o(1)$ as $N\to \infty$.
\end{theorem}

According to Definition \ref{Def5.1} of the  $\varepsilon$-Nash equilibrium, it suffices to show
\begin{equation*}
    \mathcal J_i(\hat u_i, \hat u_{-i})\leq \inf_{u_i\in \mathcal U_c^i}\mathcal J_i(u_i, \hat u_{-i})+\varepsilon.
\end{equation*}
Since the classical method of proving the $\varepsilon$-Nash equilibrium via $L^2$ error estimates alone is inadequate, a new method is applied subsequently. We will identify sufficiently close upper and lower bounds of $\inf_{u_i\in \mathcal U_c^i}\mathcal J_i(u_i, \hat u_{-i})$. Subsequently, we can show that the set of decentralized strategies $\hat u$ is an $\varepsilon$-Nash equilibrium.
\begin{proof}[Proof of Theorem \ref{thmepsilon}]
\textbf{Step 1.} Provided that we can show $\mathcal J_i(\hat u_i, \hat u_{-i})\leq \widetilde C$ for some fixed constant $\widetilde C$ independent of $N\ge \hat N$, then it is sufficient to consider all alternative control laws $u_i\in {\mathcal U}^i_c$ that satisfy
\begin{equation}\label{the boundedness of J}
    \mathcal J_i(u_i, \hat u_{-i})\leq \widetilde C.
\end{equation}
Recalling the definition  of $\Lambda_T(x, y, u)$ in \eqref{Lambda} and using  H\"older's inequality, we obtain
\begin{equation*}
    \begin{aligned}
        \mathcal J_i(\hat u_i, \hat u_{-i})&\!=\! \mathbb E\big[e^{\gamma\Lambda_T(\hat x_i, \hat x_i^{(N)},\hat u_i)}\big]\\
        &\!\leq\! \big(\mathbb E\big[e^{\gamma(1+\delta)\Lambda_T(\bar x_i, z_i, \bar u_i)}\big]\big)^{\frac{1}{1+\delta}}\big(\mathbb E\big[e^{\frac{\gamma(1+\delta)}{\delta}[\Lambda_T(\hat x_i,\hat x_i^{(N)},\hat u_i)-\Lambda_T(\bar x_i,z_i,\bar u_i)]}\big]\big)^{\frac{\delta}{1+\delta}},
    \end{aligned}
\end{equation*}
where $(\bar x_i, \bar u_i)$ is specified by \eqref{decentralized control law}-\eqref{destate} and $z_i=z_{I_i^*}$ is specified by \eqref{CCR}. Selecting $\delta_N=\frac{\gamma \sqrt[4]{\frac{1}{N}+\varepsilon_N^2}}{1-\gamma\sqrt[4]{\frac{1}{N}+\varepsilon_N^2}}$, we see that $\delta_N\rightarrow 0$ as $N\rightarrow \infty$. We will use the fact that $(\mathbb E|X|^{1+\delta})^{\frac{1}{1+\delta}}$ approaches $\mathbb E |X|$ as $\delta\rightarrow 0$, if there exists a positive $\delta_0$ such that $\mathbb E|X|^{1+\delta_0}<\infty$. By Theorem \ref{baru}, we have
\begin{equation*}
    \mathbb E\big[\exp\big\{\gamma\Lambda_T(\bar x_i, z_i, \bar u_i)\big\}\big]=\mathbb E\big[\exp\big(\gamma\big\{\xi_{i}^{\top}\Pi(0)\xi_{i}+2\xi_{i}^{\top}S_{I_i^*}(0)+r_{I_i^*}(0)\big\}\big)\big].
\end{equation*}
In conjunction with (H1) and the uniform boundedness of $(\Pi,S_{I_i^*},r_{I_i^*})$, it follows that $\mathbb E[e^{\gamma(1+\delta_0)\Lambda_T(\bar x_i, z_i, \bar u_i)}]<\infty$ for some positive constant $\delta_0$. Thus
\begin{equation}\label{11959}
    \lim_{N\rightarrow \infty}\Big(\mathbb E\Big[\exp\big\{\gamma(1+\delta_N)\Lambda_T(\bar x_i, z_i, \bar u_i)\big\}\Big]\Big)^{\frac{1}{1+\delta_N}}=\mathbb E\Big[\exp\big(\gamma \Lambda_T(\bar x_i,z_i,\bar u_i)\big)\Big].
\end{equation}
Now we take $\Theta=\frac{\gamma(1+\delta_N)}{\delta_N}=\frac{1}{ \sqrt[4]{\frac{1}{N}+\varepsilon_N^2}}$, and we can verify that condition \eqref{Thetacondition2} holds. By Lemma \ref{lem34}, we get
\begin{equation}\label{119511}
    \mathbb E\Big[\exp\Big\{\frac{\gamma(1+\delta_N)}{\delta_N}\Big|\Lambda_T(\hat x_i, \hat x_i^{(N)},\hat u_i)-\Lambda_T(\bar x_i, z_i, \bar u_i)\Big|\Big\}\Big]= 1+o(1) .
\end{equation}
Then combining \eqref{11959} with \eqref{119511}, it holds that
\begin{equation}\label{upper}
    \begin{aligned}
        \mathcal J_i(\hat u_i, \hat u_{-i})\leq J_i(\bar u_i)[1+o(1)],
    \end{aligned}
\end{equation}
where $\bar u_i$ is given by \eqref{decentralized control law}, and $o(1)\rightarrow 0$ as $N\rightarrow \infty$. Here, $J_i(\bar u_i)=J_{I_i^*}(\bar u_{I_i^*})$, as given by \eqref{optimalcost1}, is uniformly bounded with respect to $N$ under (H1).
Hence, there exists a fixed constant $\widetilde C$ such that $\sup_{1\leq i\leq N}\mathcal J_i(\hat u_i, \hat u_{-i})\le \widetilde C$ indeed holds for all $N\ge \hat N$.
In view of (H3) and \eqref{the boundedness of J}, we have
\begin{equation}\label{bound of eu}
  \mathbb E\big[e^{\gamma\int_0^T\|u_i\|_R^2\mathrm{d}t}\big]\leq \mathcal J_i(u_i, \hat u_{-i})\leq  \widetilde C,
\end{equation}
which implies that, for all sufficiently large $N$, there exists a constant $\widetilde C_1$ (depending on $\widetilde C$ and the uniform upper bound of $R$),
\begin{equation}\label{star48}
      \mathbb E\int_0^T|u_i|^2\mathrm{d}t\leq \frac{\widetilde C_1}{\gamma}.
\end{equation}

\textbf{Step 2.}
For $u_i\in {\mathcal U}_c^i$ satisfying ${\mathcal J}_i(u_i, \hat u_{-i})\le \widetilde C$, 
we aim to find a lower bound of $\mathcal J_i(u_i^{cl}, \hat u_{-i})$, where $u_i^{cl}$ appears in  \eqref{perturstate} and now additionally satisfies \eqref{bound of eu}-\eqref{star48}.

Let $X$ and $Y$ be two random variables.
For any given positive constant $c$, using H\"older's inequality, we have $\mathbb E[e^{cY}]\leq (\mathbb E[e^{\frac{c(1+\delta')}{\delta'}(Y-X)}])^{\frac{\delta'}{1+\delta'}}(\mathbb E[e^{c(1+\delta')X}])^{\frac{1}{1+\delta'}}$ for every $ \delta'>0$. Taking $c=\frac{1}{1+\delta'}$, it holds that $\mathbb E[e^X]\geq {(\mathbb E[e^{\frac{1}{1+\delta'}Y}])^{1+\delta'}}{(\mathbb E[e^{\frac{1}{\delta'}(Y-X)}])^{{-\delta'}}}$. Based on this inequality, we have
\begin{equation}\label{116513}
    \begin{aligned}
        \mathbb E\big[e^{\gamma \Lambda_T(x_i, x_i^{(N)},u_i^{cl})}\big]\!\geq \!\frac{(\mathbb E[e^{\frac{\gamma}{1+\delta'}\Lambda_T(x_i^{\infty},z_i,u_i^{cl})}])^{1+\delta'}}{(\mathbb E[e^{\frac{\gamma}{\delta'}[\Lambda_T(x_i^{\infty},z_i,u_i^{cl})-\Lambda_T( x_i, x^{(N)}_i, u_i^{cl})]}])^{{\delta'}}},
    \end{aligned}
\end{equation}
where the control $u_i^{cl}$ is the same as in \eqref{perturstate}. Note that the alternative control law $u_i$ satisfies \eqref{bound of eu}-\eqref{star48}. { Then in view of Lemma \ref{lemma56}, the estimate in Lemma \ref{lem34} after replacing $\hat u_i$ by $u_i^{cl}$ remains valid (see \eqref{718545}), which guarantees that the numerator and the denominator in \eqref{116513} are not simultaneously $0$ or $\infty$.}

To estimate the numerator in \eqref{116513}, we introduce a new auxiliary control problem (ACP), which considers the further optimization of $u_i^{cl}$ appearing in the numerator of \eqref{116513}.

\textbf{Problem (ACP).} Find $\bar u^{\delta'}_{i}\in \mathcal U_d^{i}$ to minimize
\begin{equation*}
    J^{\delta'}_i(u_i)\!=\! \mathbb E\bigg[\exp\Big(\frac{\gamma}{1+\delta'}\Big\{\int_0^T\big[\|x_i\!-\!\Gamma z_i\|_Q^2\!+\!\|u_i\|_R^2\big]\mathrm{d}t\!+\!\|x_i(T)\!-\!\Gamma_f z_i(T)\|_{Q_f}^2\Big\}\Big)\bigg],
\end{equation*}
for $\delta'>0$, subject to
\begin{equation*}
    \mathrm{d}x_i=(Ax_i+Bu_i+Dz_i)\mathrm{d}t+\sigma\mathrm{d}w_i, \quad x_i(0)=\xi_i,
\end{equation*}
where $z_i=z_{I_i^*}$ is specified by \eqref{CCR}.

By the method  of Section \ref{sec:design}, we obtain the optimal control law
\begin{equation*}
    \bar u^{\delta'}_i=-R^{-1}B^{\top}\Pi^{\delta'}\bar x^{\delta'}_i-R^{-1}B^{\top}S_i^{\delta'},
\end{equation*}
where $\bar x^{\delta'}_i$ satisfies
\begin{equation*}
    \mathrm{d}\bar x^{\delta'}_i\!=\![(A\!-\!BR^{-1}B^{\top}\Pi^{\delta'})\bar x^{\delta'}_i\!-\!BR^{-1}B^{\top}S^{\delta'}_i\!+\!Dz_i]\mathrm{d}t\!+\!\sigma \mathrm{d}w_i,\quad \bar x^{\delta'}_i(0)\!=\!\xi_i.
\end{equation*}
Here, $\Pi^{\delta'}$ is given by the following Riccati equation
\begin{equation}\label{Riccatidelta}
    \begin{aligned}
      &\dot{\Pi}^{\delta'}\!+\!\Pi^{\delta'} A\!+\!A^{\top}\Pi^{\delta'}\!-\!\Pi^{\delta'}\Big(BR^{-1}B^{\top}\!-\!\frac{{2}{\gamma}}{1\!+\!\delta'}\sigma\sigma^{\top}\Big)\Pi^{\delta'}\!+\!Q\!=\!0,     \quad \Pi^{\delta'}(T)\!=\!Q_f,
    \end{aligned}
  \end{equation}
   and $S_i^{\delta'}$ satisfies the ODE
\begin{equation}\label{Sdelta}
    \begin{aligned}
      &\dot{S}^{\delta'}_i\!+\!\Big(A^{\top}\!-\!\Pi^{\delta'} BR^{-1}B^{\top}\!+\!\frac{{2}{\gamma}}{1\!+\!\delta'}\Pi^{\delta'}\sigma\sigma^{\top}\Big)S_i^{\delta'}\!-\!(Q\Gamma\!-\!\Pi^{\delta'} D) z_i\!=\!0,
    \end{aligned}
  \end{equation}
  with $S_i^{\delta'}(T)=-Q_f{\Gamma}_fz_{i}(T)$. By (H4), we have $BR^{-1}B^{\top}-\frac{2\gamma}{1+\delta'}\sigma\sigma^{\top}\geq 0$ for each $t\in [0,T]$. Then Riccati equation \eqref{Riccatidelta} admits a unique solution $\Pi^{\delta'}\in C([0,T];\mathbb S^n)$ under (H2)-(H4). Furthermore, ODE \eqref{Sdelta} also admits a unique solution $S^{\delta'}_i\in C([0,T];\mathbb R^n)$.

  The optimal cost is given by
  \begin{equation}\label{optimalcost2}
       J^{\delta'}_i(\bar u^{\delta'}_i)\!=\!\mathbb E\Big[\exp\Big(\frac{\gamma}{1\!+\!\delta'}\big\{\xi_{i}^{\top}\Pi^{\delta'}(0)\xi_{i}+2\xi_{i}^{\top}S_i^{\delta'}(0)+r_{i}^{\delta'}(0)\big\}\Big)\Big],
  \end{equation}
  where $r_i^{\delta'}$ satisfies
  \begin{equation}\label{ODEdelta}
    \begin{aligned}
      \dot{r}_i^{\delta'}\!-\!(S_i^{\delta'})^{\top}\Big(BR^{-1}B^{\top}\!-\!\frac{{2}{\gamma}}{1\!+\!\delta'}\sigma\sigma^{\top}\Big)S_i^{\delta'} \!+\!2z^{\top}_{i} D^{\top}S_i^{\delta'}\!+\!z^{\top}_{i}\Gamma^{\top} Q\Gamma z_{i}
      \!+\!{\rm Tr}(\sigma\sigma^{\top}\Pi^{\delta'})\!=\!0,
    \end{aligned}
  \end{equation}
with $r_i^{\delta'}(T)=z^{\top}_{i}(T){\Gamma}^{\top}_fQ_f{\Gamma}_fz_{i}(T)$. Under (H2)-(H4), ODE \eqref{ODEdelta} admits a unique solution $r_i^{\delta'}\in C([0,T];\mathbb R)$.

We proceed to estimate the denominator on the right-hand side of \eqref{116513}. For \eqref{pstate}-\eqref{perturstate}, by Lemma \ref{lemma56} and H\"older's inequality,  it holds that
{\allowdisplaybreaks
     \begin{align}
         &\mathbb E\bigg[\exp\Big\{\Theta T\Big(\sup_{0\leq t\leq T}|  x_i(t)- x^{\infty}_i(t)|^2\Big)\Big\}\bigg]
         \!\leq \!\Big(\mathbb E\Big[\exp\Big\{\frac{\bar C_1 \Theta }{N^2} \Big|\sum_{j=1}^N g_{ij}^N(\xi_j-\mathbb E\xi_j) \Big|^2\Big\}\Big]\Big)^{\frac{1}{7}}    \label{eTTxx}\\
         &\ \ \times\!\Big(\mathbb E\Big[\exp\Big\{\frac{\bar C_1 \Theta }{N^4} \Big(\sum_{i=1}^N\Big|\sum_{j=1}^N g_{ij}^N(\xi_j-\mathbb E\xi_j) \Big|\Big)^2\Big\}\Big]\Big)^{\frac{1}{7}}
          \Big(\mathbb E\Big[\exp\Big\{\frac{\bar C_1 \Theta}{N^2}\int_0^T |u_i^{cl}|^2\mathrm{d}t\Big\}\Big]\Big)^{\frac{1}{7}} \notag\\
          &\ \ \times\!\Big( e^{\bar C_1 \Theta(\frac{1}{N}+\varepsilon^2_N)} \Big)^{\frac{1}{7}} \Big(\mathbb E\Big[\exp\Big\{\frac{\bar C_1 \Theta}{N^4}\Big(\sup_{0\leq t\leq T}\sum_{i=1}^N\Big|\int_0^t \sum_{j=1}^Ng_{ij}^N\sigma \mathrm{d}w_j\Big|\Big)^2\Big\}\Big]\Big)^{\frac{1}{7}}
          \notag\\
         &\ \ \times\! \Big(\mathbb E\Big[\exp\Big\{ \frac{\bar C_1 \Theta}{N^2} \Big(\sup_{0\leq t\leq T}\Big|\!\int_0^t \! \sigma \mathrm{d}w_i\Big|\Big)^2\Big\}\Big]\Big)^{\frac{1}{7}} \Big(\mathbb E\Big[\exp\Big\{\frac{\bar C_1 \Theta}{N^2}\Big(\sup_{0\leq t\leq T}\Big|\int_0^t \sum_{j=1}^Ng_{ij}^N\sigma \mathrm{d}w_j\Big|\Big)^2\Big\}\Big]\Big)^{\frac{1}{7}}
         .\notag
     \end{align}
 }\noindent
where $\bar C_1>0$ is a constant depending on $\bar C$ and $T$.  Noticing \eqref{bound of eu}-\eqref{star48}, for all deterministic positive function $\Theta=f(N)$, $N\in \mathbb N_+$, satisfying $f(N)=o(1/\sqrt{\frac{1}{N}+\varepsilon_N^2})$, we have
\begin{equation*}\begin{aligned}
    &\sup_{N\geq \hat N}\mathbb E\Big[\exp\Big\{\frac{2\bar C_1 \Theta}{N^2}\int_0^T |u_i^{cl}|^2\mathrm{d}t\Big\}\Big]\leq \widetilde C_2,\!\quad \!
      \frac{\bar C_1 \Theta}{N^2}\int_0^T |u_i^{cl}|^2\mathrm{d}t\xrightarrow[N\rightarrow \infty]{\text{in probability}}  0,
\end{aligned}\end{equation*}
which by Vitali convergence theorem implies that
\begin{equation*}
    \begin{aligned}
        \mathbb E\Big[\exp\Big\{\frac{\bar C_1 \Theta}{N^2}\int_0^T |u_i^{cl}|^2\mathrm{d}t\Big\}\Big] =1+o(1).
    \end{aligned}
\end{equation*}
By the method of proving Lemma \ref{lem3389}, we use \eqref{eTTxx} to obtain
\begin{equation} \label{Eeeo1o1}
    \begin{aligned}
        \mathbb E\big[e^{\Theta \int_0^T| x_i - x^{\infty}_i |^2\mathrm{d}t}\big]&=1+o(1),\quad \text{and} \quad
        \mathbb E\big[e^{\Theta \int_0^T| x^{(N)}_i -z_i |^2\mathrm{d}t}\big]=1+o(1).
    \end{aligned}
\end{equation}
By use of \eqref{Eeeo1o1}, we follow the proof of Lemma \ref{lem34} to establish
\begin{equation}\label{718545}
    \begin{aligned}
        &\mathbb E\Big[\exp\Big\{\Theta\big|\Lambda_T(x_i^{\infty},z_i,u_i^{cl})\!-\!\Lambda_T(x_i,x_i^{(N)},u_i^{cl})\big|\Big\}\Big]=1+o(1).
    \end{aligned}
\end{equation}

Now, we select $\delta_N'=\frac{1}{\gamma}\sqrt[4]{\frac{1}{N}+\varepsilon_N^2}$ and $\Theta=\frac{\gamma}{\delta_N'}=\frac{1}{ \sqrt[4]{\frac{1}{N}+\varepsilon_N^2}}$, and it follows from \eqref{718545} that
\begin{equation}\label{110517}
    \Big(\mathbb E\Big[\exp\Big\{\frac{\gamma}{\delta_N'}\big|\Lambda_T(x_i^{\infty},z_i,u_i^{cl})-\Lambda_T( x_i,  x^{(N)}_i, u_i^{cl})\big|\Big\}\Big]\Big)^{{\delta_N'}}=1+o(1).
\end{equation}
Combining \eqref{116513}, \eqref{optimalcost2} with \eqref{110517}, we have
  \begin{equation}\begin{aligned}\label{lower}
      \mathcal J_i(u_i^{cl}, \hat u_{-i})&\!=\!  \mathbb E\big[e^{\gamma\Lambda_T( x_i,  x_i^{(N)},u_i^{cl})}\big] \!\geq\!  \frac{[J_i^{\delta_N'}(\bar u_i^{\delta_N'})]^{1+\delta_N'}}{1+o(1)}\!=\!  [J_i^{\delta_N'}(\bar u_i^{\delta_N'})]^{1+\delta_N'}[{1\!+\!o(1)}] .
  \end{aligned}\end{equation}

\textbf{Step 3.} We proceed to compare the perturbation of the cost functionals. By the continuous dependence of solutions on parameters in ODEs, we have
\begin{equation*}
    \begin{aligned}
        \lim_{N\rightarrow \infty}(\Pi^{\delta_N'},S_i^{\delta_N'},r_i^{\delta_N'})=(\Pi,S_{I_i^*},r_{I_i^*}),
    \end{aligned}
\end{equation*}
uniformly on $[0,T]$, which implies (recalling \eqref{optimalcost1}, \eqref{optimalcost2})
\begin{equation*}
    \begin{aligned}
        \lim_{N\rightarrow \infty} J_i^{\delta_N'}(\bar u_i^{\delta_N'})=J_i(\bar u_i).
    \end{aligned}
\end{equation*}
Noticing the upper bound \eqref{upper} and lower bound \eqref{lower} of $\mathcal J_i(u_i^{cl}, \hat u_{-i})$, for all sufficiently large $N$, we have
\begin{equation*}
    J_i(\bar u_i)[1+o(1)]\leq  \inf_{u_i\in \mathcal U_c^i}\mathcal J_i(u_i, \hat u_{-i})\leq  \mathcal J_i(\hat u_i, \hat u_{-i})\leq J_i(\bar u_i)[1+o(1)],
\end{equation*}
where $J_i(\bar u_i)\times o(1)\rightarrow 0$ as $N\rightarrow \infty$, uniformly with respect to $i$. Then we have
\begin{equation}\label{110519}
    \lim_{N\rightarrow \infty}\sup_{1\leq i\leq N}\Big|\inf_{u_i\in \mathcal U_c^i}\mathcal J_i(u_i, \hat u_{-i})-\mathcal J_i(\hat u_i, \hat u_{-i})\Big|=0,
\end{equation}
which completes the proof of the theorem.
\end{proof}

\begin{remark}\label{remark47}
i) The exponentiated cost structure here makes the performance estimate more difficult than in LQG MFGs with risk-neutral costs for which standard $L^2$ error estimates are adequate \cite{CH2021,HCM2007}.

ii) If the initial states are independent Gaussian random variables with uniformly bounded second moments, we can still establish the $\varepsilon$-Nash equilibrium theorem. In this case, the compactness condition in (H1) is not fulfilled, but we can directly estimate \eqref{810515} using the method in \cite[pp. 39-42]{MP1992} instead of Hoeffding's inequality.
\end{remark}

\section{Numerical example}\label{sec:numerical}

In this section, we provide a numerical example. We first  compute the function $z_{\alpha}$, which is specified by \eqref{CCR} or equivalently by \eqref{CC}.  Here we will use \eqref{CC} to facilitate the numerical method developed below. Inspired by \cite{GCH2023}, we decouple the  forward-backward equation \eqref{CC} by the following asymmetric operator Riccati equation
\begin{equation}\label{operatorRiccati}
  \left\{
  \begin{aligned}
    &\dot{\bf P}\!+\!{\bf P}(\mathbb A\!+\![D{\bf G}])\!+\!\mathbb A^{\top}{\bf P}\!+\![{2}{\gamma}\Pi\sigma\sigma^{\top}\mathbb I]{\bf P} \!-\!{\bf P}[BR^{-1}B^{\top}{\bf G}]{\bf P}\!-\![(Q\Gamma\!-\!\Pi D)\mathbb I]\!=\!0,\\
    &{\bf P}(T)\!=\!-[Q_f\Gamma_f \mathbb I].
  \end{aligned}
  \right.
\end{equation}
In analogue to \cite[Proposition 8]{GCH2023}, we can show that the operator Riccati equation \eqref{operatorRiccati} admits a unique mild solution (see \cite[Lemma 1]{GCH2023}) if $C_{\Xi}<1$, where $C_{\Xi}$ is defined by \eqref{CXi}.

Next, we will derive the spectral decomposition of this operator Riccati equation, which allows us to approximate the operator Riccati equation by use of  some finite dimensional matrix valued Riccati equations. By \cite[Proposition 6]{GCH2023}, we have the following spectral composition of Riccati equation \eqref{operatorRiccati}:
\begin{equation*}
  {\bf P}(t)
  =[P^{\bot}(t)\mathbb I]+\sum_{\ell\in \mathcal I_{\lambda}}[(P^{\ell}(t)-P^{\bot}(t)){\bf f}_{\ell}{\bf f}_{\ell}^{\top}].
\end{equation*}
where $\{{\bf f}_{\ell}\}_{{\ell}\in \mathcal I_{\lambda}}$ denotes the orthonormal eigenfunctions of ${\bf G}$; $\lambda_{\ell}$ denotes the eigenvalue of ${\bf G}$ corresponding to ${\bf f}_{\ell}$; and  $\mathcal I_{\lambda}$ is the index multiset for all the nonzero eigenvalues of ${\bf G}$. For further details on the eigenvalues, see Example \ref{exa}.

Moreover, $P^{\bot}$ and $P^{\ell}$ are respectively given by
\begin{equation}\label{Pb}\left\{
  \begin{aligned}
    &\dot{P}^{\bot}\!+\!{P}^{\bot}(A\!-\!BR^{-1}B^{\top}\Pi)\!+\!(A\!-\!BR^{-1}B^{\top}\Pi)^{\top}P^{\bot}\!+\!{2}{\gamma}\Pi\sigma\sigma^{\top}P^{\bot} \!-\!(Q\Gamma\!-\!\Pi D)\!=\!0,\\
    &P^{\bot}(T)\!=\!-Q_f\Gamma_f,
  \end{aligned}\right.
\end{equation}
and
\begin{equation}\label{P1}
  \left\{
  \begin{aligned}
    &\dot{P}^{\ell}\!+\!P^{\ell}(A\!-\!BR^{-1}B^{\top}\Pi\!+\!\lambda_{\ell}D)\!+\!\big(A\!-\!BR^{-1}B^{\top}\Pi\!+\!{2}{\gamma}\Pi\sigma\sigma^{\top}\big) P^{\ell}\\
    &\qquad\quad\! -\!\lambda_{\ell}P^{\ell}BR^{-1}B^{\top}P^{\ell}\!-\!(Q\Gamma\!-\!\Pi D)\!=\!0,\\
    &P^{\ell}(T)\!=\!-Q_f\Gamma_f.
  \end{aligned}
  \right.
\end{equation}

\begin{figure}[!t]
\centering
\subfloat[Sinusoidal graphon]{\label{graphon}\includegraphics[width=0.45\linewidth]{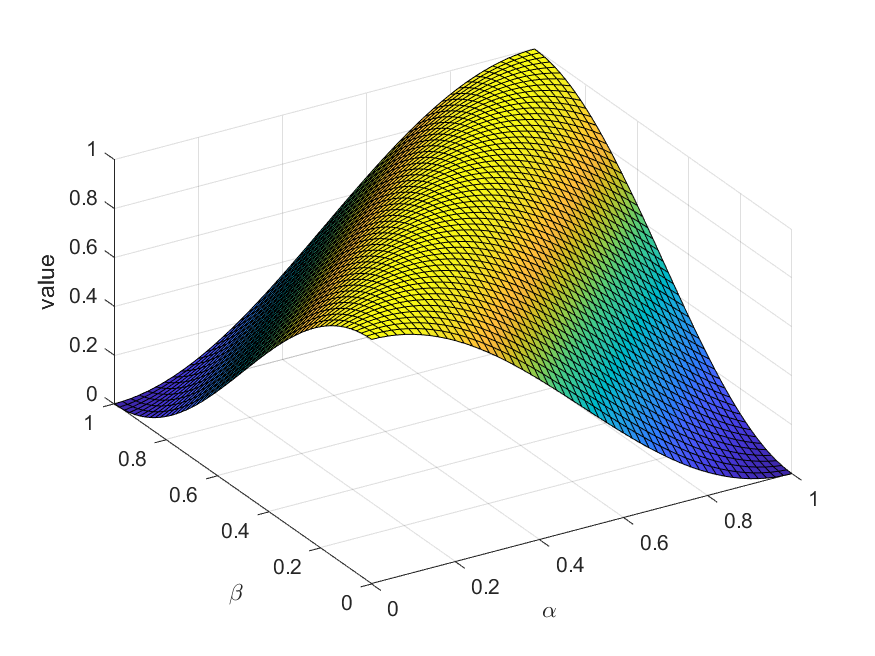}}
\subfloat[Graphon section $g_{\alpha}$]{\label{graphonsection}\includegraphics[width=0.45\linewidth]{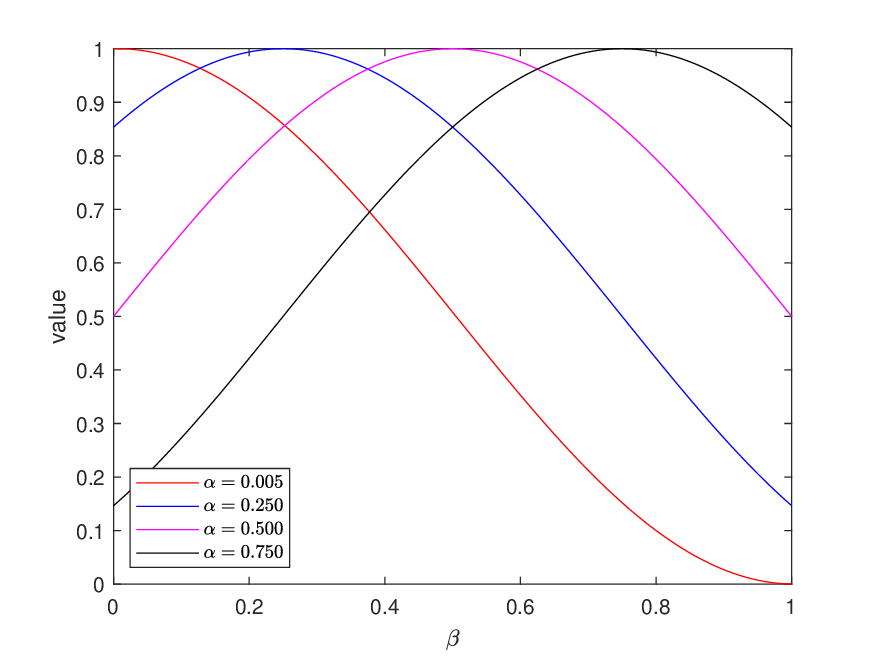}}
\caption{   Graphon and its section. }
\label{fig11}
\end{figure}
\begin{figure}[!t]
 \centering
 \includegraphics[width=0.6\linewidth]{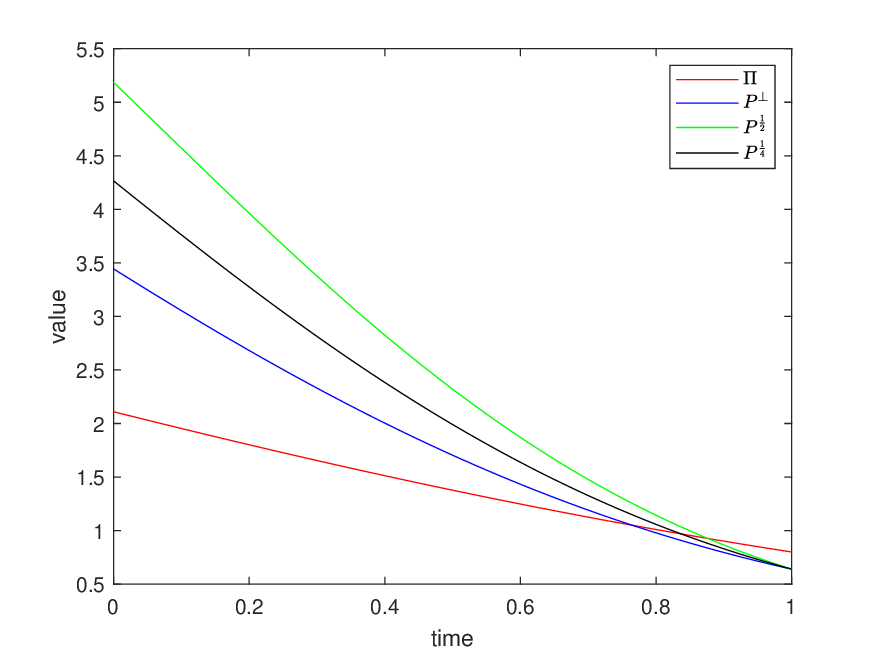}\\
 \caption{The solutions $(\Pi, P^{\perp}, P^{\frac{1}{2}}, P^{\frac{1}{4}})$ to Riccati equations.}\label{Riccati}
\end{figure}
\begin{figure}[!t]
\centering
\subfloat[State $\bar x_{\alpha}$]{\label{statep}\includegraphics[width=0.4\linewidth]{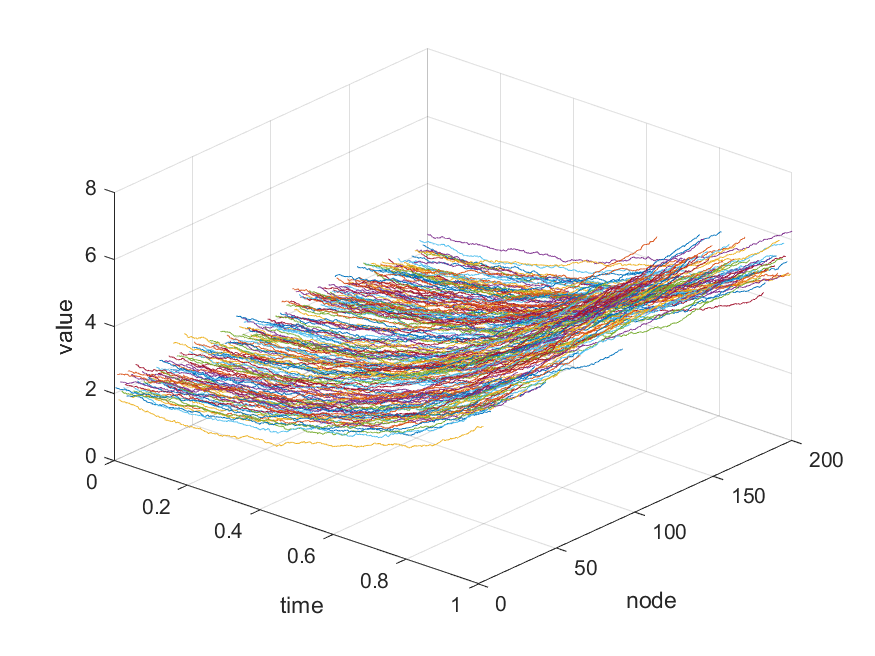}}
\subfloat[Control $\bar u_{\alpha}$]{\label{control}\includegraphics[width=0.4\linewidth]{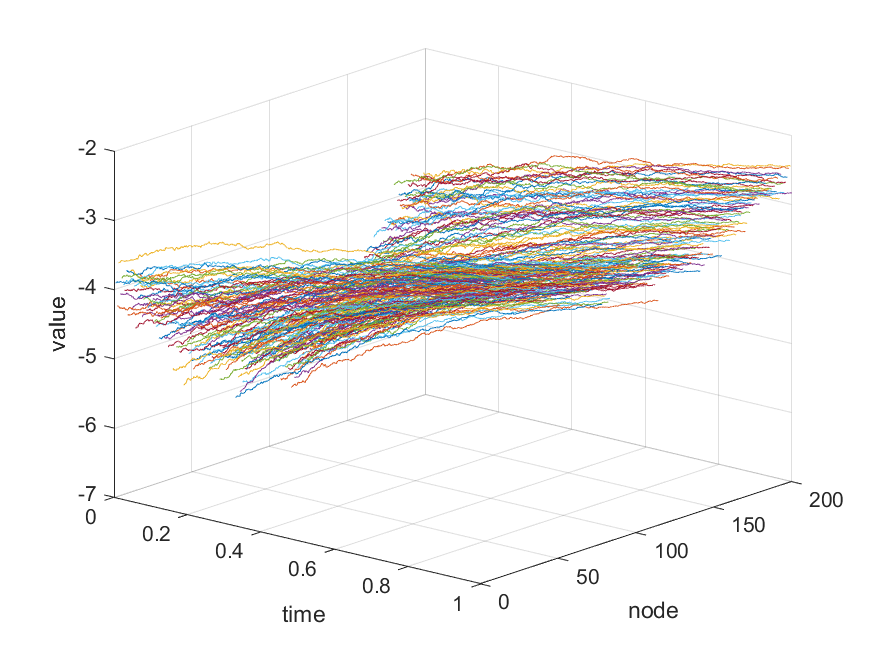}}
\caption{  The trajectories of state and control
for $200$ different $\alpha$ values under the graphon limit. }
\label{fig1}
\end{figure}
\begin{figure}[!t]
\centering
\subfloat[State $\bar x_{\alpha}$]{\label{state4}\includegraphics[width=0.45\linewidth]{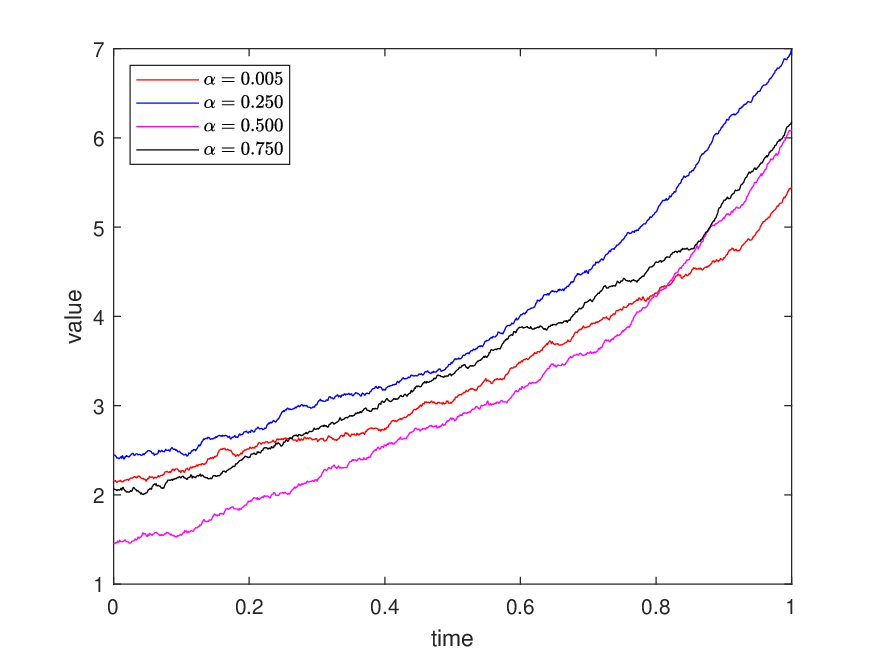}}
\subfloat[Control $\bar u_{\alpha}$]{\label{control4}\includegraphics[width=0.45\linewidth]{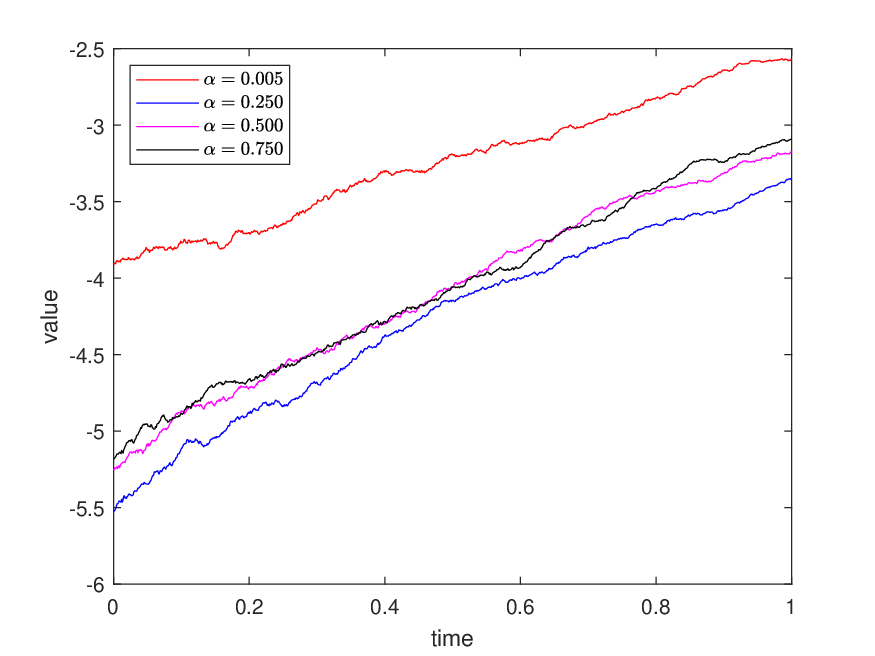}}
\caption{The trajectories of the state   and control with $4$ different $\alpha$ values.}
\label{fig4}
\end{figure}
\begin{figure}[!t]
\centering
\subfloat[Graphon weighted mean
state $z_{\alpha}$]{\label{zfig}\includegraphics[width=0.45\linewidth]{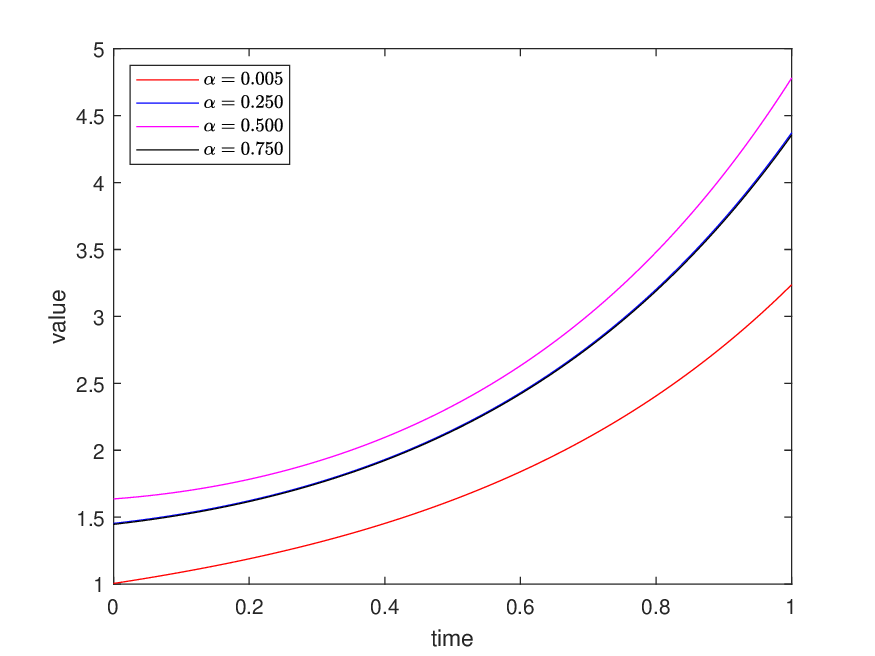}}
\subfloat[ODE solution $S_{\alpha}$]{\label{Sfig}\includegraphics[width=0.45\linewidth]{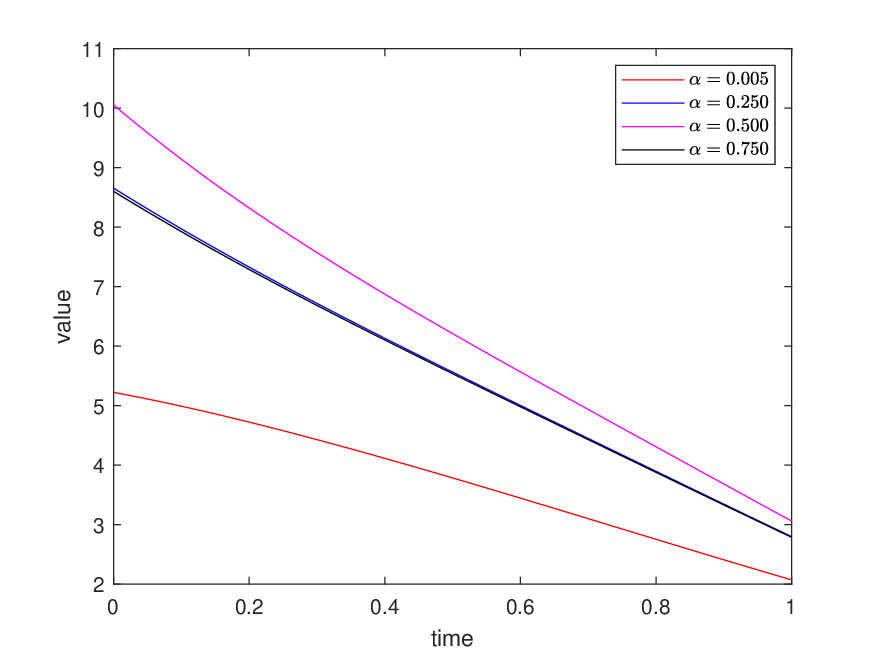}}
\caption{ The solutions to the GMFG equation system \eqref{CC}.}
\label{fig22}
\end{figure}

For illustration, we now consider a sinusoidal graphon as follows
\begin{equation}\label{g}
	{\bf G}(\alpha,\beta)=\cos^2\Big(\frac{\pi}{2} (\alpha-\beta)\Big),\quad \alpha,\beta\in [0,1],
\end{equation}
which is illustrated in Fig.~\ref{graphon}. For this sinusoidal graphon ${\bf G}$, we can show that the normalized eigenfunctions are ${\bf f}_1=1$, ${\bf f}_2=\sqrt{2}\cos \pi(\cdot)$ and ${\bf f}_3=\sqrt 2\sin \pi(\cdot)$ with eigenvalues $\lambda_1=\frac{1}{2}$, $\lambda_2=\lambda_3=\frac{1}{4}$. We take parameters  $T=1$, $\gamma=0.3$, $A=0.5$, $B=0.6$, $D=2$, $\sigma=0.5$, $Q=0.3$, $\Gamma=2$, $R=1.5$, $Q_f=0.8$, $\Gamma_f=-0.8$. The numerical solutions  of $(\Pi, P^{\bot}, P^{\frac{1}{2}}, P^{\frac{1}{4}})$ to Riccati equations \eqref{Phi} and \eqref{Pb}-\eqref{P1} are given in Fig.~\ref{Riccati}. We partition $[0,1]$ uniformly into $200$ subintervals, and take their left endpoints as the values of $\alpha$ and $\beta$. The initial values $\xi_{\alpha}$ are taken from Gaussian distribution $\mathcal N(2,0.1)$. Then the state process $\bar x_{\alpha}$ in \eqref{destate} with $200$ different $\alpha$ values in the infinite population model and the corresponding control $\bar u_{\alpha}$ are shown in  Fig.~\ref{statep}-Fig.~\ref{control}. In addition, to fully capture the heterogeneity among agents across different nodes, we plot the graphon section $g_{\alpha}$ for various values of $\alpha$ in Fig.~\ref{graphonsection}, and display  the corresponding state and control trajectories in Fig.~\ref{state4}-Fig.~\ref{control4}. As shown in Fig.~\ref{graphonsection},  the curves of graphon sections $g(0.25,\cdot)$ and $g(0.75,\cdot)$ are symmetric about the line $\alpha=0.5$. Since the initial state means are the same, then the curves of $(z_{0.25},S_{0.25})$ and $(z_{0.75},S_{0.75})$ coincide in Fig.~\ref{zfig}-Fig.~\ref{Sfig}.

\section{Conclusion}\label{sec:conclusion}
This paper investigates a class of risk-sensitive GMFGs, for which decentralized strategies are determined from the GMFG equation system as a family of fully coupled forward-backward equations indexed by the nodal parameter. The unique solvability of the GMFG equation system is established using both the fixed point method and the method of continuity.

For performance analysis of the resulting decentralized strategies,
novel exponential type error estimates are employed to establish an $\varepsilon$-Nash equilibrium theorem.

Our model does not include common noise in the agent dynamics since we are primarily interested in decentralized state feedback strategies while such strategies are generally inadequate for the search of an asymptotic Nash equilibrium when common noise appears.

\bibliographystyle{amsplain}
\bibliography{references}

\end{document}